\documentclass[letterpaper]{scrartcl}

\usepackage{amsmath, amsthm, amssymb}
\usepackage{amssymb}
\usepackage[alphabetic]{amsrefs}
\usepackage{hyperref}
\usepackage[shortlabels]{enumitem}
\usepackage[margin=3.1cm]{geometry}

\setkomafont{disposition}{\bfseries}

\newtheorem{theorem}{Theorem}[section]
\newtheorem{lemma}[theorem]{Lemma}
\newtheorem{corollary}[theorem]{Corollary}
\newtheorem{proposition}[theorem]{Proposition}

\theoremstyle{remark}
\newtheorem{remark}[theorem]{Remark}

\numberwithin{equation}{section}

\newcommand{\Cstar}{\mathrm{C}^*}

\DeclareMathOperator{\Lin}{span}

\newcommand{\N}{\mathbb{N}}

\newcommand{\R}{\mathbb{R}}
\newcommand{\C}{\mathbb{C}}
\newcommand{\Id}{\mathrm{Id}}
\newcommand{\G}{\mathrm{G}}
\newcommand{\U}{\mathrm{U}}
\newcommand{\V}{\mathrm{V}}
\newcommand{\Aut}{\mathrm{Aut}}

\newcommand{\sa}{\mathrm{sa}}

\newcommand{\X}{\mathrm{X}}
\newcommand{\Y}{\mathrm{Y}}
\newcommand{\A}{\mathcal{A}}
\newcommand{\Li}{\mathcal{L}}
\newcommand{\Ad}{\mathrm{Ad}}
\newcommand{\ad}{\mathrm{ad}}

\title{Normal subgroups of invertibles and of unitaries in  a  $\Cstar$-algebra}
\author{Leonel Robert}
\date{}
\begin{document}
	\maketitle
\begin{abstract}
We investigate the normal subgroups of the groups of invertibles  and unitaries  in the connected component of the identity.
By relating  normal subgroups to  closed two-sided ideals we obtain a  ``sandwich condition" describing all the closed normal subgroups both 
in the  invertible  and in the the unitary case.  We use this to prove a conjecture by Elliott and 
R{\o}rdam: in a simple $\Cstar$-algebra, the group of approximately inner automorphisms induced by unitaries in the connected component of the identity is topologically simple. Turning to non-closed subgroups, we show, among other things, that in simple unital $\Cstar$-algebra the commutator subgroup of the group of invertibles in the connected component of the identity
is a simple group modulo its center. A similar result holds for unitaries under a mild extra assumption.  
\end{abstract}

\section{Introduction}
We investigate below the normal subgroups of the group of invertibles and of the group of unitaries of a $\Cstar$-algebra. We confine ourselves to the connected component of the identity. Let  $\G_A$ denote the group of invertibles connected to the identity  and  $\U_A$ the unitaries connected to the identity, $A$ being the ambient $\Cstar$-algebra.
The problem of calculating  the normal subgroups of  $\G_A$ and $\U_A$ 
goes to back to Kadison's papers \cite{kad1,kad2,kad3}, where the (norm) closed normal subgroups of $\G_A$ and $\U_A$,  for  $A$ a von Neumann algebra factor, were  fully described. 
Over the years various authors have returned to this problem, focusing on the case that $A$ is  a simple $\Cstar$-algebra, and assuming further structural properties on $A$. 
To wit, the factor case is thoroughly discussed in \cite{delaHarpe2}; in \cite{delaHarpe-Skandalis3}  de la Harpe and Skandalis describe the normal subgroups of $\G_A$ and $\U_A$ for $A$ simple, unital and AF; in \cite{thomsen} Thomsen extends their results to certain simple AH C*-algebras; in \cite{elliott-rordam}  Elliott and R{\o}rdam describe the closed normal subgroups of $\U_A$ for $A$ simple,  of real rank zero, stable rank one,  and with strict comparison of projections; in \cite{ng-ruiz} Ng and Ruiz  
do the same for $A$ simple, unital, exact, and $\mathcal Z$-stable. 
The methods used in these works rely heavily on the matrix or matrix-like structure of the algebras under study. This permits the use of special similarities such as transvections and involutions (modeled after the classical case of linear groups).
In this paper we use a different approach to the problem: we rely on the theory of Lie ideals of 
$\Cstar$-algebras and on the properties of the exponential map. This yields results of greater generality. Notably,
we avoid assuming any kind of  matrix-like structure on the $\Cstar$-algebra. A direct link between normal subgroups and Lie ideals is given by the fact that the closure of the linear span of a normal subgroup is a Lie ideal. More useful to us is  that the closure of the linear span of the commutators $ha-ah$, with $h$ ranging in a normal subgroup and $a$ in $A$, is a Lie ideal. 
The exponential map, on the other hand,   plays the crucial role of bringing the profits of our analysis in the additive setting back into the multiplicative one. With these methods we obtain a full description of the closed normal subgroups of $\G_A$ and $\U_A$. By refining our methods we are also able to say something about non-closed normal subgroups. 
That this approach could bare fruit in the study of the normal subgroups of the invertibles   of a ring was put forward--quite explicitely--by Herstein in \cite{herstein-banach}. The results of this paper are confirmation that Herstein's proposal 
can be followed through in the context of $\Cstar$-algebras.

Let us introduce some notation. 
Given $x,y\in A$,
let $[x,y]=xy-yx$ and $(x,y)=xyx^{-1}y^{-1}$ (provided that $x$ and $y$ are invertible). Applied to sets,  $[X,Y]$  stands for the linear span of $[x,y]$ with $x\in X$ and $y\in Y$; 
$(X,Y)$ stands  for the subgroup generated by $(x,y)$,
with $x\in X$ and $y\in Y$. We prove below the following theorem:
 
\begin{theorem}\label{commequal}
	Let $H$ be a closed normal subgroup of $\G_A$. Let $I$ denote the closed 
two-sided ideal generated by $[H,A]$. 
	Then $\overline{(H,\G_A)}=\overline{(\G_I,\G_A)}$. 
\end{theorem}

The analogous theorem for Lie ideals, on which we rely to prove this theorem, is obtained  
by Bre\v{s}ar, Kissin and Shulman in \cite[Theorem 5.27]{BKS}.
Theorem \ref{commequal} yields a  description of all closed normal subgroups of $\G_A$
in terms of  normal subgroups associated to closed two-sided ideals of $A$.  To see this, given a closed two-sided ideal $I$ define 
\begin{align}\label{defNI}
\G_{A,I}=\{a\in \G_A\mid (a,\G_A)\subseteq \overline{(\G_I,\G_A)}\}.
\end{align}
This is clearly a normal subgroup of $\G_A$. (By Lemma \ref{NI}, $\G_{A,I}$ is simply the preimage  of the center of $\G_{A/I}$ under the quotient map.)
It follows from Theorem  \ref{commequal} that	 for any closed normal subgroup $H$ of $\G_A$ there exists a closed two-sided ideal $I$ (namely,  the one generated by $[H,A]$) such that
	\[
	\overline{(\G_I,\G_A)}\subseteq H\subseteq \G_{A,I}.
	\] 
Conversely, any  subgroup $H\subseteq \G_A$ lying in between 
	$\overline{(\G_I,\G_A)}$ and $\G_{A,I}$
	is automatically  normal.  This kind of ``sandwich condition'' describing normal subgroups is well-known in the study of normal subgroups of general linear groups
	over rings (where a matrix structure for the ring is certainly available); see \cite{bass},\cite{vaserstein}.

The analogue of Theorem \ref{commequal} for $\U_A$ is equally valid (Theorem \ref{unitarycommequal}) and obtained by the same methods. We  use it  to investigate the  normal subgroups of the group of approximately inner automorphisms induced by unitaries in $\U_A$. Let us denote this group by $\V_A$ and endow it with the topology of pointwise convergence
in norm.
 We again find an analog of Theorem \ref{commequal} for $\V_A$. Hence, its closed normal subgroups are described by a sandwich condition (this description is somewhat simplified by the fact that $\V_A$ is topologically perfect). It follows that if  $A$ is a simple $\Cstar$-algebra then  $\V_A$   is a topologically simple group. This answers a  conjecture by Elliott and R{\o}rdam from \cite{elliott-rordam}.

The methods used to prove the theorems discussed so far can be adapted to  non-closed normal subgroups. In this case we must rely on results on non-closed Lie ideals.  
We also need to look closer at the properties of the exponential map in a small neighbourhood of the identity.  More specifically, we rely on the existence of certain elements linked to the first Kashiwara-Vergne equation, as  developed by Rouvi\`{e}re in \cite{rouviere}.
Equipped with these tools we prove the following theorem:
\begin{theorem}\label{HGAGAintro}
	Let $A$ be a unital $\Cstar$-algebra. Let $H$ be a subgroup of $\G_A$ normalized by $(\G_A,\G_A)$.  Suppose that $[H,A]$ generates $A$ as a closed two-sided ideal. Then $(\G_A,\G_A)\subseteq H$. Moreover, in this case 
there exists  $[A,A]\subseteq L\subseteq A$, additive subgroup, such that
\[
H=\{e^{b_1}\cdots e^{b_n}\mid \sum_{i=1}^n b_i\in L\}.
\] 
\end{theorem}	
The assumption that $[H,A]$ generates $A$ as a closed two-sided ideal  may be rephrased as  saying that $H$ is non-central in any non-zero quotient of $A$, or roughly put, that $H$ is ``sufficiently non-central".
In the case of a simple unital $A$, Theorem \ref{HGAGAintro} implies that a subgroup normalized by $(\G_A,\G_A)$  is either contained in the center or contains $(\G_A,\G_A)$; in particular, $(\G_A,\G_A)$ is simple modulo its center.  This is the mutiplicative analogue of a well-known theorem of Herstein on Lie ideals (\cite[Theorem 1.12]{hersteinbook}).

This paper is organized as follows: In Section \ref{closedGA} we investigate the closed normal subgroups of $\G_A$, proving among other results Theorem \ref{commequal} stated above. In Section \ref{closedUA} we prove analogous results on $\U_A$ and $\V_A$. In Section \ref{nonclosedlie} we revisit some of our work from \cite{lie} on Lie ideals, in preparation for the next section. In Section \ref{nonclosedGA} we investigate the non-closed normal subgroups of $\G_A$, proving among other results Theorem \ref{HGAGAintro} stated above. In Section \ref{nonclosedUA} we do the same for $\U_A$. However, in this case the arguments do not run as smoothly; we need to assume the existence of full orthogonal projections in the algebra. Finally, in the Appendix we revisit certain manipulations of the Campbell-Baker-Hausdorff formula linked  to the first  Kashiwara-Vergne equation. We follow closely the exposition of these results in \cite{rouviere}.

\section{Closed normal subgroups of invertibles}\label{closedGA}
Let's start by fixing some  notation and recalling some useful facts. Let $A$ be a $\Cstar$-algebra. Let $A^{\sim}$ denote its minimal unitization. 
Let  $\mathrm{GL}(A)$  denote the group of invertible elements of $A$. For non-unital $A$,   $\mathrm{GL}(A)$ is defined as the invertible elements in  $A^\sim$ of the form $1+x$, with $x\in A$. We regard  $\mathrm{GL}(A)$ as a topological group under the norm topology of $A$. 
Our focus is on the connected component of the identity of $\mathrm{GL}(A)$.  This group, sometimes denoted by $\mathrm{GL^0}(A)$, will be denoted here  by $\mathrm{G}_A$.

Let $X$ be  a  subset of $A$. We denote by $\Lin(X)$ the linear span of $X$, by
$\Id(X)$  the closed two-sided ideal of $A$ generated by  $X$. 
 Let $x,y\in A$. We denote by $[x,y]$ the commutator  $xy-yx$  and by $(x,y)$ the multiplicative commutator $xyx^{-1}y^{-1}$  (provided that $x$ and $y$ are invertible).
We extend this notation to sets: If $X,Y\subseteq A$ then  $[X,Y]$ denotes the subspace  $\Lin(\{[x,y]\mid x\in X,y\in Y\})$ and  
$(X,Y)$  the subgroup  generated by the set $\{(x,y)\mid x,\in X,y\in Y\}$   (provided that $X$ and $Y$ are sets of invertibles).
We will make frequent use of the identity
\begin{equation}\label{breakprod}
(xy,a)=x(y,a)x^{-1}\cdot (x,a),
\end{equation}
or rather, of its corollary that $(x_1x_2\cdots x_n,a)$ is contained in any normal subgroup  containing $(x_1,a),\ldots,(x_n,a)$.

Let $e^z$ denote the exponential of $z\in A$. Recall that $\{e^a\mid a\in A\}$
is a generating set of $\G_A$. We  make use below of various well known properties of the exponential map some of  which we now review:
For all $a,b\in A$ we have  
\begin{equation}\label{CBH1}
\log(e^{a}e^{b})-(a+b)\in \overline{[A,A]}
\end{equation}
for $a,b\in A$ such that $\|a\|+\|b\|< \frac {\log 2} 2$. This is a consequence of  the convergence of the Campbell-Baker-Hausdorff formula (\cite{dynkin}). 
Yet another consequence   of the Campbell-Baker-Hausdorff formula 
that we use frequently is Trotter's formula:
\begin{equation}\label{trotter}
e^{a+b}=\lim_n \Big(e^{\frac a n}e^{\frac b n}\Big)^n.
\end{equation}
Observe that it implies that  $e^{a+b}$ belongs to the closed subgroup generated by 
$e^{\frac a n}$ and $e^{\frac b n}$ for $n=1,2\dots$.

A subspace $L\subseteq A$ is called a Lie ideal if $[L,A]\subseteq A$. We make use throughout the paper  of  various results  on Lie ideals of $\Cstar$-algebras. We will recall them as  needed, but 
we mention two here which  feature crucially in various arguments. The first,  
\cite[Theorem 1]{herstein}, says that if  $L$ is a Lie ideal of $A$ and $x\in A$ is such that $[x,[x,L]]=0$ then $[x,L]=0$.  This theorem is valid more generally for semiprime rings without 2-torsion. The second, \cite[Lemma 1.6]{lie}, says that  if $L$ is a closed Lie ideal of $A$ such that $L\subseteq \overline{[A,A]}$ and $\Id(L)=\Id([L,A])$
then $L=\overline{[\Id(L),A]}$. This is  a  convenient version of the theorem by Bre\v{s}ar, Kissin, and Shulman \cite[Theorem 5.27]{BKS},  relating the Lie ideals of a $\Cstar$-algebra to its closed two-sided ideals (of which Theorem \ref{commequal} from  the introduction is the multiplicative analogue). 

For each closed two-sided ideal $I$ of $A$ we regard $I^\sim$  as a subalgebra of $A^\sim$, and, in this way, $\G_I$ as a closed  normal subgroup of $\G_A$.
Our first goal is to prove Theorem \ref{commequal} from the introduction. Before, we prove a couple of lemmas.

\begin{lemma}\label{idealnormal}
	Let $I$ be a closed two-sided ideal of $A$. Then the set $e^{[I,A]}$ generates  	$\overline{(\G_I,\G_A)}$ as a topological group.	
\end{lemma}
\begin{proof}
Let  us first show that $e^c\in \overline{(\G_I,\G_A)}$ for all $c\in [I,A]$.
Let $c\in [I,A]$.
Since  $\G_I$ spans $I^\sim$, $[I,A]=[\G_I,A]$. Moreover,
from the formula  $[h,a]=h(ah)h^{-1}-ah$
 we see that the elements of $[\G_I,A]$ are expressible as sums of elements 
of the form $hah^{-1}-a$, with $h\in \G_I$ and $a\in A$. Thus, $c$ is a finite sum of such elements. By
 Trotter's formula \eqref{trotter}, to prove that $e^c\in  \overline{(\G_I,\G_A)}$ 
 we may reduce ourselves to the case  that $c=hah^{-1}-a$ for some $h\in \G_I$ and $a\in A$.  Trotter's formula again yields that $e^{hah^{-1}-a}$   is the limit of $(he^{a/n}h^{-1}e^{-a/n})^n=(h,e^{a/n})^n\in (\G_I,\G_A)$, as desired.

Let us now show that $\overline{(\G_I,\G_A)}$ is contained in the closed subgroup generated by $e^{[I,A]}$. Call this subgroup $H$. Observe that $H$ is normal  by the  invariance of $e^{[I,A]}$ under conjugation by $\G_A$. Let $g\in \G_I$ and $a\in \G_A$ and let us 
show that $(g,a)\in H$. Writing  $g$
	as a product of exponentials in $\G_I$ and using \eqref{breakprod}  we reduce ourselves to the case that $g=e^b$ for some $b\in I$. We can also assume that
	$b$ is very small. Let us choose $b$ small enough---depending only on $a$, which we regard as fixed---so that  $c:=\log ((a,e^b))$ is defined.  Since $(a,e^b)=e^{aba^{-1}}e^{-b}$ is of the form $1+x$  with $x\in I$, we have $c\in I$. Moreover, choosing $b$ small enough we can apply \eqref{CBH1} and get  
$c-(aba^{-1}-b)\in \overline{[A,A]}$. Thus, $c\in \overline{[A,A]}\cap I$. But 
	\begin{equation}\label{mierseq}
	\overline{[A,A]}\cap I=\overline{[I,A]},
	\end{equation} 
	by \cite[Lemma 1]{miers}.  Hence, $c\in \overline{[I,A]}$, and so  $e^c$
belongs to the closure of $e^{[I,A]}$, as desired.
\end{proof}

\begin{lemma}\label{NI}	Let $I$ be a closed two-sided ideal of $A$. 
Let $g\in \G_A$. If $(g,\G_A)\subseteq I^\sim$ then $(g,\G_A)\subseteq \overline{(\G_I,\G_I)}$.
\end{lemma}
\begin{proof}
Let $g\in \G_A$ be as in the statement of the lemma and let $a\in \G_A$. To prove that $(g,a)\in \overline{(\G_I,\G_A)}$ it suffices to choose $a$ from a generating set of $\G_A$ (by \eqref{breakprod}).  Let us pick $a=e^{b}$, with $b$   small so that $c:=\log((g,e^b))$ is defined. We thus want to show that $e^c\in \overline{(\G_I,\G_A)}$. From 
\[
\lim_{t\to 0}\frac{(g,e^{tb})-1}{t}=gbg^{-1}-b,
\] 
and that the left hand side belongs to $I^\sim$, we deduce that $gbg^{-1}-b\in I^\sim$. Hence $[g,b]\in I^\sim$, which in turn implies that $[g,b]\in I$, for it is well known that a commutator cannot be  a non-zero scalar multiple of the identity. It follows that  $c\in I$. Moreover, choosing  $b$ small enough we also have  that $c\in \overline{[A,A]}$  by \eqref{CBH1}. So $c\in \overline{[I,A]}$ by \eqref{mierseq}.  Hence, $e^c\in \overline{(\G_I,\G_A)}$ by Lemma \ref{idealnormal}.  
	\end{proof}

\begin{remark}
The previous lemma shows that the group  $\mathrm N_I$ defined in \eqref{defNI} (in the introduction) admits a somewhat more concrete description
	as the preimage by the quotient map $A^\sim\to A^\sim/I$ of the center of  
	$\G_{A/I}$.	
\end{remark}

\begin{proof}[Proof of Theorem \ref{commequal}]
We deduce from the definition of $I$  that $(H,\G_A)\subseteq  I^\sim$ (since $H$ becomes  central after taking the quotient by $I$). The inclusion  $(H,A)\subseteq \overline{(\G_I,\G_A)}$ now  follows from Lemma \ref{NI}.  

	Let us prove that $(\G_I,\G_A)\subseteq \overline{(H,A)}$. By Lemma \ref{idealnormal}, it suffices to 
	show that $e^{[I,A]}\subseteq  \overline{(H,\G_A)}$.   Let
	\[
	L=\overline{\Lin\{hah^{-1}-a\mid h\in H, a\in A\}}.
	\]
We claim that $L=\overline{[I,A]}$. Before proving this claim, let us use it  to complete the proof of the theorem. Let $c\in [I,A]$. Since $c\in L$,
it can be approximated by a sum of 
elements of the form $hah^{-1}-a$, with $h\in H$ and $a\in A$.   By Trotter's formula, to prove that $e^c\in  \overline{(H,A)}$  we may reduce ourselves to the case that $c= hah^{-1}-a$ with  $h\in H$ and $a\in A$. Applying  Trotter's formula once more we see that 
  $e^{hah^{-1}-a}$  is the limit of  $(e^{h\frac a n h^{-1}}e^{-\frac a n})^n=(h,e^{\frac a n})^n\in (H,\G_A)$. Thus, $e^c\in \overline{(H,A)}$, as desired.

Let us prove our claim that $L=\overline{[I,A]}$. From the identities
 $hah^{-1}-a=[h,ah^{-1}]$ 
	and   $[h,a]=h(ah)h^{-1}-ah$ we deduce   that   $L=\overline{[H,A]}$. Thus,
we must   show that $\overline{[H,A]}=\overline{[I,A]}$. 
Since $H$ is normal, the closed subspace $\overline{\Lin(H)}$ is invariant under conjugation by elements of $\G_A$.   It is thus a Lie ideal of $A$ (see \cite[Theorem 2.3]{marcoux-murphy}). 
By \cite[Theorem 5.27]{BKS} applied to this  Lie ideal we have that
$\overline{[H,A]}=\overline{[J,A]}$,
where $J=\Id([\Lin(H),A])=\Id([H,A])=I$. This proves our claim.
\end{proof}

\begin{corollary}[Cf. \cite{herstein0}*{Theorem 4}]\label{dichotomy}
	Let $A$ be a simple $\Cstar$-algebra. Let $H$ be  a closed normal subgroup of $\G_A$. Then either $H$ is contained in the center of $\G_A$ or 
$\overline{(\G_A,\G_A)}\subseteq H$.
\end{corollary}	
\begin{proof}
	Assume that $H$ is non-central. Then $\Id([H,A])\neq 0$, and so $\Id([H,A])=A$ by the simplicity of $A$. By the previous theorem, $\overline{(\G_A,\G_A)}=\overline{(H,\G_A)}\subseteq H$.
\end{proof}

For the remainder of this section we explore further properties of the normal subgroups of $\G_A$ relying on Theorem \ref{commequal} and the circle of ideas used to prove it.

Let $N_2$  denote the set of square zero elements of $A$. That is, $N_2=\{x\in A\mid x^2=0\}$. It is shown in  
\cite[Corollary 2.3]{lie} that $\overline{\Lin(N_2)}=\overline{[A,A]}$. In Theorem \ref{U2} below we prove the multiplicative analogue of this result.

\begin{lemma}\label{N2comm}
Let $x\in N_2$. Then
$1+x=(u,v)$  for some $u,v\in (\G_A,\G_A)$.
\end{lemma}
\begin{proof}
Let $M=\|x\|$. It suffices to prove the result in the universal $\Cstar$-algebra generated by a square zero element of norm at most $M$. Thus,  we may assume that $A=M_2(C_0(0,M])$ and that $x=\left(\begin{smallmatrix} 0 & t\\ 0 & 0\end{smallmatrix}\right)$, where $t$ denotes the identity function on $(0,M]$. 
 Note that, since the group of invertibles of $M_2(C_0(0,M]^\sim$ is connected, all the invertible elements that we  construct below are in the connected component of the identity.

Let $f=t^{\frac 1 2}$ and  $g=(1+t^{\frac 1 2})^{\frac 1 2}$, so that 
$f\cdot (g^2-1)=t$. Then
\[
\begin{pmatrix}
1 & t\\
0 & 1
\end{pmatrix}=
\left(
\begin{pmatrix}
g & 0\\
0 & g^{-1}
\end{pmatrix},
\begin{pmatrix}
1 & f\\
0 & 1
\end{pmatrix}
\right).
\]
The element $\left(\begin{smallmatrix} 1 & f\\ 0 & 1\end{smallmatrix}\right)$
is a commutator for the same reason that  $\left(\begin{smallmatrix} 1 & t\\ 1 & 0\end{smallmatrix}\right)$ is one. It remains 
 to express  $\left(\begin{smallmatrix} g & 0\\ 0 & g^{-1}\end{smallmatrix}\right)$ as a commutator.  Let $h=((1+t^{\frac 1 2})^{\frac 1 2}-1)^{\frac 1 2}$, so that $1+h^2=g$.  A straightforward computation shows that
\[
\begin{pmatrix}
g  & 0\\
0 & g^{-1}
\end{pmatrix}=
\left(
\begin{pmatrix}
g^{\frac 1 2} & h\\
0 & g^{-\frac 1 2}
\end{pmatrix},
\begin{pmatrix}
g^{-\frac 1 2} & 0\\
h & g^{\frac 1 2}
\end{pmatrix}
\right).\qedhere
\] 
\end{proof}

A  group $G$ is called perfect if $G=(G,G)$. If $G$ is a topological group then it is called  topologically perfect if $G=\overline{(G,G)}$. 

\begin{theorem}\label{U2}
The set $1+N_2$ generates $\overline{(\G_A,\G_A)}$ as a topological group. The group  $\overline{(\G_A,\G_A)}$ is topologically perfect.
\end{theorem}

\begin{proof}
	Let $U_2$ denote the closed subgroup of $\G_A$ generated by  $1+N_2$. Since $1+N_2$ is invariant under similarities, $U_2$ is  normal.  Theorem \ref{commequal} applied  to $U_2$ gives us that  $\overline{(U_2,\G_A)}=\overline{(\G_I,\G_A)}$, where  $I=\Id([U_2,A])$. 	 Theorem \ref{commequal} applied to the group $\G_A$ gives us that
$\overline{(\G_A,\G_A)}=\overline{(\G_J,\G_A)}$, where $J=\Id([\G_A,A])$. 
Let us show that $I=J=\Id([A,A])$. The equality $[\G_A,A]=[A,A]$ is clear since $\G_A$ spans $A^\sim$.  Thus, $J=\Id([A,A])$.  To see that $I=\Id([A,A])$  we first use the identity
	\[
	[xy,a]=[x,ya]+[y,ax],
	\]
to deduce that   $\overline{[U_2,A]}=\overline{[1+N_2,A]}=\overline{[N_2,A]}$. Hence,   	$I=\Id([N_2,A])$. But $\Id([N_2,A])=\Id([A,A])$, by \cite[Corollary 2.3]{lie}.  So $I=\Id([A,A])$.
Now we have
\[
\overline{(\G_A,\G_A)}=\overline{(U_2,\G_A)}\subseteq U_2.
\]
  On the other hand,   $U_2\subseteq \overline{((\G_A,\G_A), (\G_A,\G_A))}$  by Lemma \ref{N2comm}. This proves the two claims of the theorem.
\end{proof}

The inclusion $(\G_A,\G_A)\subseteq \overline{(\G_A,\G_A)}$ is often proper (e.g., take $A$ an irrational rotation $\Cstar$-algebra). Thus,  $\overline{(\G_A,\G_A)}$ may fail to be perfect in the algebraic sense. We will show in Section \ref{nonclosedGA} that $(\G_A,\G_A)$ is perfect whenever $A$ is unital and without 1-dimensional representations.

Recall that a subgroup $H\subseteq \G_A$ is called subnormal if there exists
a finite chain of subgroups $H\subseteq H_1\cdots\subseteq H_n=\G_A$, each normal in the next.
\begin{theorem}\label{subnormal}
The closed subnormal subgroups of $\G_A$ are normal.
\end{theorem}
\begin{proof}
Let $H$ be closed and subnormal. Taking the closure of all the subgroups in
the chain of subgroups starting at $H$ and ending in $\G_A$, we may assume that they are all closed.
It is then clear that it suffices to show that if $H\subseteq G\subseteq \G_A$ are closed subgroups such that $H$ is normal in $G$ and $G$ is normal in  $\G_A$ then $H$
is normal in $\G_A$. We prove this next.
Let $I=\Id([H,A])$.
We clearly have that $(H,\G_A)\subseteq I^\sim$, since $H$ is central in the quotient by $I$. By Lemma \ref{NI},  
\begin{enumerate}
\item[(1)]
$(H,\G_A)\subseteq \overline{(\G_I,\G_A)}$. 
\end{enumerate}
On the other hand,  since $G$ is normal in $\G_A$, we have   by Theorem \ref{commequal} that $\overline{(G,\G_A)}=\overline{(\G_J,\G_A)}$, where  $J=\Id([G,A])$. Since $I\subseteq J$,  
$\overline{(\G_I,\G_A)}\subseteq G$, and since $H$ is normal in $G$,  
\begin{enumerate}
\item[(2)]
$H$ is invariant under conjugation by $\overline{(\G_I,\G_A)}$.
\end{enumerate}
We will deduce from (1) and (2) that $H$ is normal in $\G_A$.  By (1),  it suffices to show that $\overline{(\G_I,\G_A)}\subseteq H$.  Moreover, since the exponentials $e^b$, with $b\in [I,A]$,
generate $\overline{(\G_I,\G_A)}$  (by Lemma \ref{idealnormal}), it suffices to show that such exponentials belong to $H$.
Let 
\[
L=\overline{\Lin\{hah^{-1}-a\mid h\in H,a\in [I,A]}\}.
\] 
We claim that $L=\overline{[I,A]}$. Before proving this claim, let us explain how to complete the proof of the theorem.  As stated above,
we want to show that $e^{c}\in H$ for $c\in [I,A]$. Approximating  $c$
by  sums of elements of the form $hah^{-1}-a$, with $h\in H$ and $a\in [I,A]$, and applying Trotter's formula,
we are reduced to showing that $e^{hah^{-1}-a}\in H$ for $h\in H$ and $a\in [I,A]$. Applying Trotter's formula once more, we are further reduced to showing that
$(h,e^a)\in H$ for all $h\in H$ and $a\in [I,A]$. This indeed holds, since  $e^a\in \overline{(\G_I,\G_A)}$ and $H$ is invariant under conjugation by
$\overline{(\G_I,\G_A)}$.

Let us now prove that $L=\overline{[I,A]}$. We first show that $L$ is a Lie ideal of $A$.
Notice that $L$ is invariant under conjugation by 
$\overline{(\G_I,\G_A)}$, since both $H$ and $[I,A]$ are. Let $a\in [I,A]$ and $l\in L$. Then $e^{ta}\in \overline{(\G_I,\G_A)}$  for all
$t\in \R$.  Thus, the elements on the left side of  
\[
\lim_{t\to 0}\frac{e^{ta}le^{-ta}-l}{t}=[a,l]
\] 
belong to $L$ for all $t\in \R$. It follows that $[L,[I,A]]\subseteq L$. Let $(e_\lambda)$ be an approximately central approximate unit for $I$. Using that $L\subseteq I$, we deduce that for all $l\in L$ and $a,b\in A$  
\[
[l,[a,b]]=\lim_\lambda [le_\lambda ,[a,b]]=\lim_\lambda [l,[ae_\lambda,b]]\in L.
\]
Thus, $[L,[A,A]]\subseteq L$. By \cite[Theorem 1.15]{lie}, $L$ is a Lie ideal of $A$.  It is clear from the definition of $L$ that $L\subseteq \overline{[A,A]}$. Thus, by \cite[Lemma 1.6]{lie}, in order to  prove that $L=\overline{[I,A]}$ 
it suffices to show that $\Id(L)=\Id([L,A])=I$. Let  $\widetilde A$ denote the quotient of $A$ by $\Id([L,A])$. Let $\widetilde H$ and $\widetilde I$  denote the images  of $H$ 
and $I$ in this quotient.
Let $h\in \widetilde H$ and $a\in [\widetilde I,\widetilde A]$.
	Then
	$h a h^{-1}-a$ is central in  $\widetilde A$ and in particular  commutes with $h$.  Hence, $h$ commutes with $(hah^{-1}-a)h=[h,a]$; i.e.,  
	$[h,[h,a]]=0$.  Since this holds for all $a\in [\tilde I,\tilde A]$, we have 
	$[h,[h,[\widetilde I,\widetilde A]]]=0$. By Herstein's \cite[Theorem 1]{herstein} applied to the Lie ideal $[\widetilde I,\widetilde A]$, we have $[h,[\widetilde I,\widetilde A]]=0$. 
In particular, $[h,[h,\widetilde I]]=0$, which implies that $[h,\widetilde I]=0$ (Herstein's theorem again).  Since 
$[h,\widetilde A]\subseteq \widetilde I$,   we get $[h,[h,\widetilde  A]]=0$, which implies that
 $[h,\widetilde A]=0$ for all $h\in \widetilde H$. That is, $[H,A]\subseteq \Id([L,A])$. Hence $I=\Id([H,A])\subseteq \Id([L,A])$. On the other hand, we clearly have that $L\subseteq I$ from the definition of $L$. Thus, $\Id(L)=\Id([L,A])=I$. This completes the proof of  $L=\overline{[I,A]}$, as claimed. 
\end{proof}

\begin{corollary}\label{essentiallysimple}
Let $A$ be  simple $\Cstar$-algebra. Then $\overline{(\G_A,\G_A)}$	
divided by its center is a topologically simple group.
\end{corollary}	
\begin{proof}
	Let $H\subseteq \overline{(\G_A,\G_A)}$  be  a closed  normal subgroup of $\overline{(\G_A,\G_A)}$ properly containing  the  center of $\overline{(\G_A,\G_A)}$. By Theorem \ref{subnormal} $H$ is also normal
	in $\G_A$. Thus, by Corollary \ref{dichotomy}, either  $\overline{(\G_A,\G_A)}\subseteq H$ or $H$ is in the center of $\G_A$. But the latter case cannot be because $H$ is not  contained in the center of $\overline{(\G_A,\G_A)}$.
	Hence, $H=\overline{(\G_A,\G_A)}$.
\end{proof}

\begin{theorem}
If  $H$ is a closed subgroup of $\G_A$ normalized by $(\G_A,\G_A)$.  
Then $H$ is normal in $\G_A$. 
\end{theorem}

\begin{proof}
Let $I=\Id([H,A])$. As in the proof of Theorem \ref{subnormal}, we deduce that
\begin{enumerate}
	\item[(1)]
$(H,\G_A)\subseteq \overline{(\G_I,\G_A)}$.	
\end{enumerate}		
Also, since $H$ is normalized by $\overline{(\G_A,\G_A)}$,
\begin{enumerate}
	\item[(2)]
	$H$ is invariant under conjugation by  $\overline{(\G_I,\G_A)}$.
\end{enumerate}		
Now proceeding as in the proof of Theorem \ref{subnormal}
	 we deduce from (1) and (2) that $H$ is a normal subgroup of $\G_A$. 	
\end{proof}

\section{Closed normal subgroups of unitaries}\label{closedUA}
Let us denote the   unitary group of $A$  by $\U(A)$. If $A$ is non-unital, the same convention  from the  case of invertibles is applied to unitaries: they are chosen in $A^\sim$ and of the form $1+x$, with $x\in A$. Our focus in this section is on  the connected component of the identity in $\U(A)$, which we  denote by $\U_A$.
The closed normal subgroups of $\U_A$ can be handled in much the same way as 
those of  $\G_A$. So here our arguments  will be more succinct, mostly pointing out the necessary adaptations.

Let $A_{\sa}$ denote the set of selfadjoint elements of $A$.  Recall that $\{e^{ih}|\mid h\in A_{\sa}\}$ is a generating set of $\U_A$.  
For each closed two-sided ideal $I$ of $A$ we regard $I^\sim$  as a subalgebra of $A^\sim$, and, in this way, $\U_I$ as a (closed, normal)
subgroup of $\U_A$.

\begin{lemma}\label{unitarygenerators}
	Let $I$ be a closed two-sided ideal of $A$. Then the set $\{e^{ic}\mid c\in [I,A]\cap A_{\sa}\}$ generates $\overline{(\U_I,\U_A)}$ as a topological group.	
\end{lemma}
\begin{proof}
Let  us first show that $e^{ic}\in \overline{(\U_I,\U_A)}$ for all $c\in [I,A]$ selfadjoint.
Let $c\in [I,A]$ be a selfadjoint.
Since  $\U_I$ spans $I^\sim$ linearly, $c\in [\U_I,A]$. Moreover,
from the formula  $[u,a]=u(au)u^*-au$
 we see that $c$ is  a sum of elements 
of the form $uau^*-a$, with $u\in \U_I$ and $a\in A$.  Taking the selfadjoint part of these elements we may assume that $a\in A_{\sa}$. As in the proof of Lemma \ref{idealnormal},  to prove $e^{ic}\in  \overline{(\U_I,\U_A)}$ we apply 
 Trotter's formula twice; first to reduce to the case  that $c=uau^*-a$ and then to see that $e^{i(uau^*-a)}$   is the limit of $(u,e^{i\frac a n})^n\in (\U_I,\U_A)$, as desired.

Let us now show that $\overline{(\U_I,\U_A)}$ is contained in the closed subgroup generated by the elements  $e^{ic}$ with $c\in [I,A]\cap A_{\sa}$. Call this group $H$. Observe that $H$ is normal  in $\U_A$ by the  invariance of its generating set under conjugation by $\U_A$. Let $u\in \U_I$ and $v\in \U_A$ and let us 
show that $(u,v)\in H$. Writing  $u$
	as a product of exponentials in $\U_I$ and using \eqref{breakprod}  we reduce ourselves to the case that $u=e^{ib}$ for some $b\in I_\sa$. We can also assume that
	$b$ is very small.  Let us choose $b$ small enough  so that  $(a,e^{ib})=e^{ic}$ for some $c\in A_{\sa}$.  As in the proof of Lemma \ref{idealnormal}, applying  \eqref{CBH1}  we find that $c\in \overline{[A,A]}\cap I=\overline{[I,A]}$.  Writing $c$ as a limit of elements
in $[I,A]$ and taking their selfadjoint parts if necessary we find that $e^{ic}$
belongs to the closure of $\{e^{id}\mid d\in [I,A]\cap A_{\sa}\}$, as desired.
\end{proof}	

\begin{lemma}\label{unitaryNI}	Let $I$ be a closed two-sided ideal of $A$. 
Let $u\in \U_A$. If $(u,\U_A)\subseteq I^\sim$ then $(u,\U_A)\subseteq \overline{(\U_I,\U_A)}$.
\end{lemma}
\begin{proof}
	Let $v\in \U_A$ and let us show that
 $(u,v)\in \overline{(\U_I,\U_A)}$. 
	By \eqref{breakprod} it suffices to choose $v$ from a generating set of $\U_A$; we thus  assume  that $v=e^{ib}$ for some $b\in A_{\sa}$. The proof now proceeds exactly as in Lemma \ref{NI}:  Choosing $b$ small enough we have that $(u,e^{ib})=e^{ic}$ for some selfadjoint $c\in \overline{[A,A]}$.
	Passing to the limit as $t\to 0$ in  $((u,e^{itb})-1)/t\in I^\sim$  we get that 
	$ubu^*-b\in I^\sim$, whence 
	$[u,b]\in I^\sim$, and further $[u,b]\in I$ (since a commutator cannot be a non-zero scalar). Hence  $c\in I\cap \overline{[A,A]}=\overline{[I,A]}$ (by \eqref{mierseq}).  Thus, 
	$(u,e^{ib})=e^{ic}\in \overline{(\U_I,\U_A)}$
	by Lemma \ref{unitarygenerators}.
	\end{proof}

\begin{theorem}\label{unitarycommequal}
	Let $H$ be a normal subgroup of $\U_A$. 
Then $\overline{(H,\U_A)}=\overline{(\U_I,\U_A)}$, where $I=\Id([H,A])$.	
\end{theorem}	

\begin{proof}
From the definition of $I$ we clearly have that
$\overline{(H,\U_A)}\subseteq I^\sim$. The inclusion $\overline{(H,\U_A)}\subseteq \overline{(\U_I,\U_A)}$ now follows from Lemma \ref{unitaryNI}. 
Let us prove that   $\overline{(\U_I,\U_A)}\subseteq \overline{(H,\U_A)}$. By Lemma \ref{unitarygenerators} it suffices to show that $e^{ic}\in \overline{(H,\U_A)}$ for all  selfadjoint $c\in \overline{[I,A]}$. Let us define
\[
L=\overline{\Lin(\{uau^*-a\mid u\in H,a\in A\})}.
\]	
We claim that $L=\overline{[I,A]}$. Before proving this, let us complete the proof of the theorem: Let $c\in \overline{[I,A]}$ be a selfadjoint. Since $c$ is in $L$  it can be approximated 
by a sum of elements of the form $uau^*-a$, with $u\in H$ and $a\in A$.
Taking the selfadjoint part of these terms if necessary, we may further assume that $a\in A_{\sa}$. By Trotter's formula, we reduce ourselves to the case that  $c=uau^*-a$, with $u\in H$ and $a\in A_{\sa}$.
But $e^{i(uau^*-a)}$, by Trotter's formula,  is the limit of $(u,e^{ia/n})^n\in (H,\U_A)$. Thus, $e^{ic}\in \overline{(H,\U_A)}$, as desired.

Let us  prove the claim that  $L=\overline{[I,A]}$. First notice that 
$L=\overline{[H,A]}$ because  of the formulas   $uau^*-a=[u,au^*]$ 
and   $[u,a]=u(au)u^*-au$. To prove that  $\overline{[H,A]}=\overline{[I,A]}$ we proceed as in the proof of Theorem \ref{commequal}: We first notice  that $\overline{\Lin(H)}$ is a Lie ideal since 
it is a closed subspace invariant   under conjugation by $\U_A$  (see \cite[Theorem 2.3]{marcoux-murphy}). By \cite[Theorem 5.27]{BKS} applied to $\overline{\Lin(H)}$ we have that
$\overline{[H,A]}=\overline{[J,A]}$,
where $J=\Id([\Lin(H),A])=\Id([H,A])=I$. This proves our claim.
\end{proof}	

\begin{corollary}
Let $A$ be a simple $\Cstar$-algebra. Let $H$ be a closed  normal subgroup of $\U_A$. Then either $H$ is contained in the center of $\U_A$ or $\overline{(\U_A,\U_A)}\subseteq H$.	
\end{corollary}	
\begin{proof}
Argue as in the proof of Corollary \ref{dichotomy} using Theorem \ref{unitarycommequal} in place of Theorem \ref{commequal}.
\end{proof}

\begin{lemma}\label{harpeskand}
Let $x\in N_2$ be such that  $\|x\|<\sqrt{\frac \pi 2}$. Then $e^{i[x^*,x]}=(u,v)$ for some $u,v\in (\U_A,\U_A)$.	
\end{lemma}	
\begin{proof}
As in the proof of Lemma \ref{N2comm}, we may assume that $A=M_2(C_0(0,M])$  and that $x=\left(\begin{smallmatrix} 0 & t\\ 0 & 0\end{smallmatrix}\right)$, where $t$  is the identity function on $(0,M]$ and $M<\sqrt{\frac \pi 2}$. The lemma now follows from \cite[Lemma 5.13]{dlHarpe-Skandalis2}.
\end{proof}	

\begin{theorem}\label{unitaryperfect}
The set $\{e^{i[x^*,x]}\mid x\in N_2\}$ generates $\overline{(\U_A,\U_A)}$
as a topological group. The group $\overline{(\U_A,\U_A)}$ is topologically perfect.
\end{theorem}	
\begin{proof}
		Let $V_2$ denote the closed subgroup of $\U_A$ generated by  $\{e^{i[x^*,x]}\mid x\in N_2\}$. By the unitary conjugation invariance of its set of generators, $V_2$  is normal. Theorem \ref{unitarycommequal} applied to $V_2$ yields that $\overline{(V_2,\U_A)}=\overline{(\U_I,\U_A)}$, where  $I=\Id([V_2,A])$.  Let us show that $I=\Id([A,A])$.  The inclusion
$I\subseteq \Id([A,A])$ is obvious from the definition of $I$. Let $\widetilde A$ denote the quotient of $A$ by $I$.
Let $x\in \widetilde A$ be a square zero element.  We can lift $x$ to a square zero element in $A$. From this we deduce that   $e^{it[x^*,x]}$ is in the center of $\widetilde A$ for all $t\in\R$.  This,  in turn, implies that  $[x^*,x]=\lim_{t\to 0}(e^{it[x^*,x]}-1)/t$ is in the center. By \cite[Lemma 5.1]{tracedim},   the selfadjoint  and skewadjoint parts of a square zero element are both of the form $[y^*,y]$, where $y$ is a square zero element. Thus,  all the square zero elements  of $\widetilde A$ are in its center.
But a commutative  $\Cstar$-algebra  contains no nonzero nilpotents. Thus, $\widetilde A$ has no nonzero square zero elements, which in turn implies that it is  commutative. Hence $\Id([A,A])\subseteq I$.  We have thus shown that
$\overline{(V_2,\U_A)}=\overline{(\U_{\Id([A,A])},\U_A)}$. On the other hand,
Theorem \ref{unitarycommequal} applied to $\U_A$ implies that $\overline{(\U_A,\U_A)}=\overline{(\U_{\Id([A,A])},\U_A)}$.  So
\[
\overline{(\U_A,\U_A)}=\overline{(V_2,\U_A)}\subseteq V_2 \subseteq \overline{((\U_A,\U_A),(\U_A,\U_A))},
\]
where we have used  Lemma \ref{harpeskand} in the righmost inclusion. These inclusions must be equalities, which proves the two claims of the theorem.
\end{proof}	

\begin{theorem}
The closed subnormal subgroups of $\U_A$ are normal.
\end{theorem}	
\begin{proof}
	The proof runs along the same lines as the proof of Theorem \ref{subnormal}, with minor modifications. As in that proof,
our purpose is to show that if $H\subseteq G\subseteq \U_A$ are closed subgroups such that $H$ is normal in $G$
and $G$ is normal in $\U_A$ then $H$ is normal in $\U_A$.
We set $I=\Id([H,A])$
	and observe that  $(H,\U_A)\subseteq I^\sim$. By Lemma \ref{unitaryNI}, we get that
\begin{enumerate}
\item[(1)]
$(H,\U_A)\subseteq \overline{(\U_I,\U_A)}$. 
\end{enumerate}
Also, 
\begin{enumerate}
\item[(2)]
$H$ is invariant under conjugation
	by $\overline{(\U_I,\U_A)}$,
\end{enumerate}
because it is invariant under conjugation by $G$ which contains
	$\overline{(\U_J,\U_A)}\supseteq \overline{(\U_I,\U_A)}$  (where $J=\Id([G,A])$).
Next we use (1) and (2) to deduce that $H$ is normal. By (1),	it  suffices to show that
	$(\U_I,\U_A)\subseteq H$. To do this, it suffices to show that $e^{ic}\in H$  for all selfadjoint $c\in [I,A]$ (by Lemma \ref{unitarygenerators}). Let us define
	\[
	L=\overline{\Lin(\{uau^*-a\mid u\in H,a\in [I,A]\})}.
	\] 
We claim that  $L=\overline{[I,A]}$. Assume that this claim is true. Let $c\in [I,A]\cap A_{\sa}$. 
Since  $c\in L$, it  is a limit of   sums of elements of the form $uau^*-u$, with $a\in [I,A]$. Taking the  selfadjoint part of these sums,  we can assume  $a\in [I,A]\cap A_{\sa}$ (since the selfadjoint part of an element in $[I,A]$ is again in $[I,A]$). We can now assume that $c$ is equal to one of these sums. Moreover, by Trotter's formula applied twice  
	we are reduced to showing that $(u,e^{ia})$ belongs to $H$, for all $u\in H$ and selfadjoint $a\in [I,A]$. This holds since $e^{ia}\in \overline{(\U_I,\U_A)}$, by Lemma \ref{unitarygenerators}, and $H$ is invariant under $\overline{(\U_I,\U_A)}$.

Finally,  let us  prove the claim that  $L=\overline{[I,A]}$. From the invariance of $L$ under conjugation by $e^{ic}$, with $c\in [I,A]\cap A_{\sa}$ we deduce that $[L,c]\subseteq L$
for all  $c\in [I,A]\cap A_{\sa}$. But $[I,A]\cap A_{\sa}$ spans $[I,A]$. Hence,
$[L,[I,A]]\subseteq L$.  We continue arguing as  in the last paragraph of the proof of Theorem \ref{subnormal} to conclude that $L=\overline{[I,A]}$. 
\end{proof}	

\begin{corollary}
	Let $A$ be a simple $\Cstar$-algebra. Then $\overline{(\U_A,\U_A)}$
	divided by its center is a topologically simple group.
\end{corollary}	
\begin{proof}
We can argue as in  Corollary \ref{essentiallysimple}.
\end{proof}

Let  $\Aut(A)$ denote the group of automorphisms  of $A$. We regard it  endowed with the topology  of pointwise convergence in norm.
Let $u\stackrel{\Ad}{\longmapsto} \Ad_u$ denote  the map associating to each unitary $u$ the inner automorphism
$\Ad_u(a): =uau^*$, for all $a\in A$. Let $\V_A$ denote the closure of the image of $\U_A$ under $\Ad$. That is,
\[
\V_A=\{\phi\in \Aut(A)\mid \Ad_{u_\lambda}\to \phi,
\hbox{ $u_\lambda\in \U_A$ for all $\lambda$}\}.
\]
For each closed two-sided ideal $I$ of $A$, let $\V_I$ denote the closure of the image of $\U_I$ under $\Ad$. 
That is,
\[
\V_I=\{\phi\in \Aut(A)\mid \Ad_{u_\lambda}\to \phi,
\hbox{ $u_\lambda\in \U_I$ for all $\lambda$}\}.
\]
Observe that $\V_I$  is a closed normal subgroup of $\V_A$. We stress that  the elements of $\V_I$ are automorphisms of $A$ rather than $I$. 

\begin{lemma}\label{multKR}
	Let $I$ be a closed two-sided ideal of $A$ and $B\subseteq A$ a separable $\Cstar$-subalgebra.
\begin{enumerate}[(i)]
	\item
	For each  $u\in \U_I$  there exist $v_1,v_2,\ldots\in (\U_I,\U_A)$ such that $[uv_n^*,b]\to 0$ for all $b\in B$. 
	\item
	The image of $ (\U_I,\U_A)$ under the map $\U(A)\stackrel{\Ad}{\longrightarrow} \Aut(A)$ is dense in $\V_I$.
\end{enumerate}	
\end{lemma}

\begin{proof}
 We work with the sequence algebra $A_\infty=\prod_{n=1}^\infty A/\bigoplus_{n=1}^\infty A$ to simplify some calculations.  We regard $A$ as a subalgebra 
 of $A_\infty$ via the diagonal embedding $a\mapsto [{(a)_{n=1}^\infty}]$ (the brackets denote the equivalence class of a sequence).

Let us first assume that $u=e^{ic}$ for some $c\in I_{\sa}$. Without loss of generality, assume that $B$ contains $c$. 
	By \cite[Lemma 6.4]{kirchberg-rordam},
	there exists $d=[(d_n)_{n=1}^\infty]\in A_\infty$ such that $c-d$  commutes with $B$ and
each $d_n$ has the form $\frac 1 N\sum_{j=1}^N (w_jcw_j^*-c)$ for some unitaries $w_1,\dots,w_N\in \U_A$. Observe from this that 
	$d_n$ is selfadjoint  and that $d_n\in [I,A]$ for all $n$.
	Set $v_n=e^{id_n}$ for $n=1,\dots$.
	These unitaries have the desired property. Indeed, on one hand $e^{id_n}\in \overline{(\U_I,\U_A)}$ for all $n$  by Lemma \ref{unitarygenerators}. Also, working in $A_\infty$, we have that 
	$e^{ic}e^{-id}=e^{i(c-d)}$, since $d$ commutes with $c$, and that $e^{i(c-d)}$  commutes with $B$, since $c-d$ commutes with $B$. 
	Thus, $e^{ic}e^{-id}$ commutes with $B$, as desired.

	Suppose now that  $u=e^{ic_1}e^{ic_2}\cdots e^{ic_m}$, with 
	$c_1,\dots,c_m\in I_{\sa}$. 
	Assume without loss of generality that $B$ contains $c_1,\dots,c_m$. For each $c_k$
	choose  selfadjoints $d_k=[(d_{k,n})_{n=1}^\infty]\in A_\infty$ such that $c_k-d_k$ commutes with $B$ and $d_{k,n}\in [I,A]$ for all $k=1,\dots,m$ and $n\in \N$. Set $v_n=e^{id_{1,n}}\cdots e^{id_{m,n}}$  for all $n\in \N$.  These unitaries have the desired property. Indeed, 
	working in $A_\infty$, we have that 
	\begin{align*}
	(e^{ic_1}\cdots e^{ic_m})(e^{-id_m}\cdots e^{-id_1}) &=e^{i(c_m-d_m)}e^{ic_1}\cdots e^{ic_{m-1}}e^{-id_{m-1}}\cdots e^{-id_1}\\
	&=e^{i(c_m-d_m)}\cdots e^{i(c_1-d_1)}.
	\end{align*}
	Since $c_k-d_k$ commutes with $B$ for all $k$, the right side commutes with $B$, as desired.
	
	(ii) Since the automorphisms  $\Ad_u$, with $u\in \U_I$, are dense in $\V_I$, it suffices to approximate them by automorphisms $\Ad_v$ with $v\in (\U_I,\U_A)$.
Let $u\in \U_I$. Let $F\subseteq A$ be a finite set and $\epsilon>0$. By (i), there exists $v\in (\U_I,\U_A)$ such that 
	$\|[uv^*,x]\|<\epsilon$ for all $x\in F$. It follows that  $\|\Ad_u(x)-\Ad_u(x)\|<\epsilon$ for all $x\in F$, as desired. 
	\end{proof}

		The following theorem may be regarded as
a description of the closed normal subgroups of $\V_A$ in much the same way that
Theorems \ref{commequal} and \ref{unitarycommequal} are for $\G_A$ and $\U_A$.

\begin{theorem}\label{autocommequal}
Let $G$ be a closed normal subgroup of $\V_A$.  Then  $\overline{(G,\V_A)}=\V_I$, where $I=\Id(\{\phi(x)-x\mid x\in A,\,\phi\in G\})$.
\end{theorem}

\begin{proof}
	Let $G$ be as in the statement of the theorem.  Let
	$H=\{u\in \U_A\mid \Ad_u\in G\}$.  Then $H$ is a  norm  closed normal subgroup of
	$\U_A$. By Theorem \ref{unitarycommequal}, $\overline{(\U_J,\U_A)}\subseteq H$, where $J=\Id([H,A])$. Hence,
	$\mathrm{Ad}_u\in G$ for all $u\in \overline{(\U_J,\U_A)}$. Lemma \ref{multKR} (ii) then implies that  $\V_J\subseteq G$. Now taking commutators with $\V_A$ we get
	$\overline{(\V_J,\V_A)}\subseteq \overline{(G,\V_A)}$. Recall now that, by Theorem \ref{unitaryperfect}, the group $\\overline{(\U_J,\U_J)}=\overline{(\U_J,\U_A)}$
	is topologically perfect.  This group is mapped by $\Ad$ as a dense subgroup of $\V_J$  (Lemma \ref{multKR} (ii)). So $\V_J$ is also topologically perfect. Hence, $\V_J=\overline{(\V_J,\V_A)}\subseteq G$. 
	
	Let us  show that $(G,\V_A)\subseteq \V_J$.  Let $\phi\in G$. From  $(\phi,\Ad_u)=\Ad_{\phi(u)u^*}$ we deduce that $\phi(u)u^*\in H$ for all $u\in \U_A$.
Since the set $\{\Ad_u\mid u\in (\U_A,\U_A)\}$ is dense in $\V_A$
it suffices to show that $(\phi,\Ad_u)\in \V_J$ for all
	$u\in (\U_A,\U_A)$.  As argued before on similar occasions, by the identity \eqref{breakprod}  we may choose $u$ from a generating set of  $(\U_A,\U_A)$. So let's assume that $u=(v,w)$, with $v,w\in \U_A$. Let $\nu=\Ad_{v^{-1}}$ and $\omega=\Ad_{w}$. We wish to show that $(\phi,(\nu^{-1},\omega))\in \V_J$. By the Hall-Witt identity,
	\[
(\phi,(\nu^{-1},\omega))^\nu\cdot 
(\nu,(\omega^{-1},\phi))^\omega\cdot
(\omega,(\phi^{-1},\nu))^\phi=1
	\]
(where $\alpha^\beta$ means $\beta\alpha\beta^{-1}$).
It is thus sufficient to show that $(\nu,(\omega^{-1},\phi))$ and $(\omega,(\phi^{-1},\nu))$ belong to $\V_J$. Since they are formally similar, let us deal with $(\nu,(\omega^{-1},\phi))$ only.  We compute that
\[
(\nu,(\omega^{-1},\phi))=(\Ad_{v^*},(\Ad_{w^*},\phi))=\Ad_{(v^*,w^*\phi(w))}.
\] 
Since $w^*\phi(w)\in H$,
\[
(w^*\phi(w),v^*)\in \overline{(H,\U_A)}=\overline{(\U_J,\U_A)}\subseteq \U_J.
\]
Thus, $\Ad_{(v^*,w^*\phi(w))}\in \V_J$, as desired.

Finally, let us prove that  the ideal $J$  agrees with the ideal $I$ from the statement of the theorem.  Recall that $[H,A]$ is spanned by the elements $uau^*-a=\Ad_u a-a$, where $u\in H$ and $a\in A$; i.e., 
where  $\Ad_u\in G$ and $a\in A$. The inclusion $J\subseteq I$ is then clear. To prove the opposite inclusion, let $\phi\in G$ and let us show that $\phi(x)-x\in J$ for all $x$.  Notice that $H$ is invariant under $\phi$, since $\phi\cdot \Ad_u=\Ad_{\phi(u)}$. So  $J$ is invariant under $\phi$. Let $\tilde \phi$ denote the automorphism induced by $\phi$ in the quotient $\widetilde A=A/J$. We wish to show that $\tilde \phi$ is the identity map.   Since $\phi(u)u^*\in H$ for all $u\in\U_A$, $\tilde\phi(u)u^*$ belongs to the center of $\widetilde A$ for all $u\in \U_{\widetilde A}$. That an approximately inner automorphism with this property must be the identity we prove in the lemma below and with this conclude the proof of the theorem. 
\end{proof}

\begin{lemma}
If $\phi$ is an approximately inner automorphism such that $\phi(u)u^*$ is in the center of $A$ for all $u\in \U_A$ then $\phi$ is the identity.
\end{lemma}
\begin{proof}
The assumption on $\phi$ implies that $u\mapsto \phi(u)u^*$ is a group homomorphism with abelian range. Hence $\phi(u)u^*=1$ for all  $u\in\overline{(\U_A,\U_A)}$; i.e., $\phi$ is the identity on $\overline{(\U_A,\U_A)}$.   From $\phi(e^{itc})=e^{itc}$ for all $c\in [A,A]\cap A_{\sa}$ and $t\in\R$ we deduce that  $\phi(c)=c$ for all $c\in [A,A]\cap A_{\sa}$. The set  $[A,A]\cap A_{\sa}$ spans $[A,A]$ so $\phi$ is the identity on $[A,A]$. Now, since $\phi$ is approximately inner,
$\phi(x)-x\in \overline{[A,A]}$. So $\phi(\phi(x)-x))=\phi(x)-x$  for all $x$; i.e., $\phi^2(x)=2\phi(x)-x$ for all $x$. Exploiting the multiplicativity of the left side we get 
$(2\phi(x)-x)(2\phi(y)-y)=2\phi(x)\phi(y)-xy$ for all $x,y$ which, after simple manipulations, implies that $(\phi(x)-x)(\phi(y)-y)=0$ for all $x,y$.
Setting $y=x^*$ we get $\phi(x)=x$ for all $x$, as desired.
\end{proof}	

  In \cite{elliott-rordam} Elliott and R{\o}rdam prove that $\V_A$ is topologically simple if $A$
  is a simple unital C*-algebra of real rank zero, stable rank one, and having strict comparison of projections. They then conjecture that $\V_A$ is topologically simple for any simple $A$. The following corollary proves this:
\begin{corollary}\label{answerER}
If $A$ is  a simple $\Cstar$-algebra then $\V_A$ is a topologically simple group.
\end{corollary}	
\begin{proof}
	Let $G$ be a closed subgroup of $\V_A$ such that $G\neq \{1\}$. By the previous theorem
$\V_I\subseteq G$ where $I=\Id(\{\phi(x)-x\mid x\in A,\,\phi\in G\})$. But $I$ is  non-zero, since $G\neq \{1\}$. So  $I=A$, by the simplicity of $A$,  and  $G=\V_A$.
\end{proof}	

\section{Non-closed Lie  ideals revisited}\label{nonclosedlie}
In order to extend the results from the previous sections to non-closed normal subgroups we rely on
results for non-closed  Lie ideals. The statement of these results (e.g., in \cite{lie}) are  not quantitative  enough to be directly used for our purposes here.  Instead,  we need to re-examine the mechanics of their proofs. We do so in this section.

In the sequel  by polynomial we always mean a polynomial in noncommuting variables with coefficients in $\C$.

Let $\pi_n(x_1,\dots,x_{2^n})$ be the polynomial defined by $\pi_0(x)=x$ and
\begin{equation}\label{pin}
\pi_{n+1}(x_1,\dots,x_{2^{n+1}})=[\pi_n(x_1,\dots,x_{2^n}),\pi_n(x_{2^{n}+1},\dots,x_{2^{n+1}})]
\end{equation}
for all $n\geq 0$.
\begin{lemma} \label{lieidentities}
	Let $a,b,c,x_1,\dots,x_8$  be noncommuting variables. 
	\begin{enumerate}[(i)]
		\item
		The polynomial	$[a[x_1,x_2]b,c]$ is expressible as a sum whose terms have either the form $[x_i,r]$ or  $[[x_i,r],s]$, where $i=1,2$  and where $r$ and $s$
		are polynomials in the variables $a,b,c,x_1,x_2$.
		
		\item
		The polynomial $[a\pi_3(x_1,\dots,x_{8})b,c]$ is expressible as a sum whose terms have either  the form
		$[x_i,[r,s]]$ or $[[x_i,[r,s]],[r',s']]$, where $i=1,\dots,8$  and where $r,s,r',s'$
		are polynomials in the variables $a,b,c,x_1,\dots,x_8$.
	\end{enumerate}		
\end{lemma}	

\begin{proof}
	(i) We have 
	\begin{align*}
		a[x_1,x_2]b &=ab[x_1,x_2]+a[[x_1,x_2],b]\\
		&=ab[x_1,x_2]+a[x_1,[x_2,b]]-a[x_2,[x_1,b]].
	\end{align*}
	Using that $x[y,z]=[y,xz]-[y,x]z$, the first and second term on the right can be further manipulated as follows (the third term is formally as the second):
	\begin{align*}
		ab[x_1,x_2] &=[x_1,abx_2]-[x_1,ab]x_2,\\
		a[x_1,[x_2,b]]&=[x_1,a[x_2,b]]-[x_1,a][x_2,b].
	\end{align*}	
	We now apply $[\cdot,c]$ on both sides  and break up the products on the right side using that
	$[xy,c]=[x,yc]+[y,cx]$. This yields the desired result.
	
	(ii) By (i), $[a\pi_3(x_1,\dots,x_8)b,c]$ is a sum of terms of the form $[\pi_2,r]$ and $[[\pi_2,r],s]$,
	where $r$ and $s$ are polynomials and where $\pi_2$ is evaluated on either $(x_1,\dots,x_4)$ or $(x_5,\dots,x_{8})$. Let us show that these terms are expressible in the form required by the lemma. This follows essentially by repeated applications of Jacobi's identity.  For the terms 
	$[\pi_2,r]$  we have that
	\begin{align*}
		[\pi_{2}(x_1,\dots,x_4),r] &=[[[x_1,x_2],[x_3,x_4]],r]\\
		&=	[[x_1,x_2],[[x_3,x_4],r]]-[[x_3,x_4],[[x_1,x_2],r]].
	\end{align*}
	Elements of the form $[[x_i,x_j],[u,v]]$, as on the right side, can be expressed as 
	\[
	[x_i,[x_j,[u,v]]]-[x_j,[x_i,[u,v]]],
	\] 
	whose terms are as required.
	On the terms $[[\pi_2,r],s]$ we first express $[\pi_2,r]$  as a sum whose terms have the form $[x_i,[x_j,[u,v]]]$, as argued above, and then  use that
	\[
	[ [ x_i, [x_j , [u,v] ] ] , s ] =[ x_i, [ [x_j,[u,v]],s] ] - [ [x_j , [u,v] ]  , [  x_i , s ] ].
	\]
	The terms on the right side are again as required.
\end{proof}

Let $A$ be a $\Cstar$-algebra. 
Let $N_2^c$ denote the subset of $N_2\subseteq A$ defined as follows:
\[
N_2^c=\{x\in A\mid \exists e,f\in A_+\hbox{ such that }fx=xe=x\hbox{ and }ef=0\}.
\]
$N_2^c$ is  dense in $N_2$ (\cite[Section 4]{lie}).

\begin{lemma}\label{sumof5}
	Let $x\in N_2^c$ and $r\in A$. Then
	\[
	[x,r]=z_1+z_2+z_3+z_4+(1+z_5)x(1-z_5),
	\]
	for some  $z_1,z_2,z_3,z_4,z_5\in N_2^c$ such that $\|z_i\|\leq \|x\|\cdot \|r\|$ 
for all $i$.	
\end{lemma}
\begin{proof}
	Here we follow closely the proof of \cite[Lemma 4.1]{lie}. If either $x=0$ or $r=0$ then the lemma is trivial.
	Let us then assume that $\|x\|=\|r\|=1$. The general case reduces to this one by rescaling.
	Let $e,f\in A_+$ be such that $fx=ex=x$ and $ef=0$. Using functional calculus, let us modify  $e$ and $f$ so that they are contractions. Also with functional calculus, applied to $e$, let us find $e_0,e_1,e_2,e_3\in A_+$ such that $e_ie_{i+1}=e_{i+1}$ for $i=0,1,2$ and 
	$xe_3=x$; we find similarly 
	$f_0,f_1,f_2,f_3\in A_+$  such that $f_if_{i+1}=f_{i+1}$ for $i=0,1,2$ and $f_3x=x$. 
	
	Let $r\in A$.
	\[
	[x,r]=(1-e_1)rx+[e_1rf_1,x]-xr(1-f_1).
	\] 
	One readily checks that $1-e_2$ and $e_3$ are units on the left and on the right respectively for the first term. 
	Similarly,  $f_3$ and $1-f_2$ are units on the left and on the right for the last term.  Thus, these terms  are contractions in $N_2^c$.
	As for the middle term, we have that
	\[
	[e_1rf_1,x]=(1+e_1rf_1)x(1-e_1rf_1)+(e_1rf_1)x(e_1rf_1)-x.
	\]
	The first term on the right has the form $(1+z)x(1-z)$, where $z\in N_2^c$ is a contraction. 
	The other two are contractions in $N_2^c$ as well. 
	The lemma is thus  proved.
\end{proof}

\begin{theorem}\label{commN2}
	Let $A$ be a unital $\Cstar$-algebra without 1-dimensional representations. Then there exists $K\in \N$ and $C>0$ such that for each $c,d\in A$ we have 
	\[
	[c,d]=\sum_{i=1}^K y_i,
	\] 
	for some  $y_1,\dots,y_K\in N_2$ such that $\|y_i\|\leq C\|c\|\cdot \|d\|$ for all $i$.
\end{theorem}
\begin{proof}
This theorem is contained  \cite{lie}*{Theorem 4.2}, except that we have added the norm bound on the elements $y_i$.

Since $A$ has no 1-dimensional representations, $A=\Id([A,A])$. We have 
	\[
\Id([A,A])=\Id([[A,A],[A,A]])=\Id([N_2,N_2])=\Id([N_2^c,N_2^c]).
	\]
Hence,  $A=\Id([N_2^c,N_2^c])$.
	Since $A$ is unital, there exist  elements $x_1,y_1,\dots,x_m,y_m\in N_2^c$ and $a_1,b_1,\dots,a_m,b_m\in A$ such that
	\begin{align}\label{1comm}
	1=\sum_{i=1}^ma_i[x_i,y_i]b_i.
	\end{align}

	Let $c,d\in A$. If $c=0$ or $d=0$ the theorem holds trivially.  So assume that $c\neq 0$ and $d\neq 0$, and then, after rescaling, that $\|c\|=\|d\|=1$. Multiplying by $c$ and taking commutator by $d$ in \eqref{1comm} we get 
	\[
	[c,d]=\sum_{i=1}^m[ca_i[x_i,y_i]b_i,d].
	\]
	By Lemma \ref{lieidentities} (i), the $i$-th term of the sum on the right is expressible as a  sum whose  terms 
	have either one of the following forms: $[x_i,r]$, $[y_i,r]$, $[[x_i,r],s]$, or $[[y_i,r],s]$.
	More explicitely,
\begin{equation}\label{morexplicit1}
[c,d]=\sum_{i=1}^m[x_i,r_i^{(1)}]+\sum_{i=1}^m[y_i,r_i^{(2)}]+
\sum_{i=1}^m\sum_{k=1}^{M_{i}}[[x_i,r_k^{(3)}],s_k^{(1)}]+
\sum_{i=1}^m\sum_{k=1}^{M_{i}}[[y_i,r_k^{(4)}],s_k^{(2)}].
\end{equation}
	The elements  $r_i^{(j)}$ and $s_{i}^{(j)}$  are all polynomials
	in $c, d,a_i,x_i,y_i,b_i$. 
	In particular, their norms are uniformly bounded for  $c$ and $d$ of norm 1 (we regard
	$a_i,x_i,y_i,b_i$ as fixed).  Therefore, it suffices to show that
	for each $x\in N_2^c$ and $r,s\in A$ such that $\|r\|,\|s\|\leq M$ the elements $[x,r]$ and $[[x,r],s]$ can be written
	as sums of elements in $N_2$ controlling the number of elements and their norms. 
	Let us consider a term of the form $[x,r]$  
	   By Lemma \ref{sumof5}, 
	   \[
	   [x,r]=z_1+z_2+z_3+z_4+(1+z_5)x(1-z_5),
	   \]
 with $z_i\in N_2$ and $\|z_i\|\leq \|x\|\cdot \|r\|$ for all $i$. Taking commutator with $s$ we get 
 \[
 [[x,r],s]=\sum_{i=1}^4 [z_i,s]+(1+z_5)[x,s'](1-z_5),
 \]
 where $s'=(1-z_5)s(1+z_5)$. By Lemma \ref{sumof5},	the terms $[z_i,s]$ for  $i=1,2,3,4$ and $[x,s']$ are again sums of five elements of $N_2$ whose norms are bounded from above.  Thus, both $[x,r]$ and $[[x,r],s]$
	can be expressed as  sums 
	of at most 25 of elements in $N_2$ with a uniform bound on their norms. 
	\end{proof}

\begin{theorem}[Cf. \cite{lie}*{Corollary 3.8}]\label{commP}
Let $A$ be a unital $\Cstar$-algebra containing a projection $p$ such that both $p$ and $1-p$ are full (i.e., they each generate $A$ as a closed two-sided ideal). Then there exists $K\in \N$
such that for any contraction $c\in A$ there exist projections $p_1,\dots,p_K\in A$
such that
\[
[c^*,c]=\sum_{i=1}^K \epsilon_ip_i,
\]
and where $\epsilon_i=\pm 1$ for all $i$.
\end{theorem}
\begin{proof}
Let $P$ denote the set of projections of $A$. By the assumption of fullness of $p$ and $1-p$,  we have that 
$A=\Id([P,A])$ (because in $A/\Id([P,A])$ the projections $p$ and $1-p$ become both central and full). Since $\Lin(P)$ is a Lie ideal (\cite[Theorem 4.2]{marcoux-murphy}),
$\Id([P,A])=\Id([P,P])$. (In general, $\Id([L,L])=\Id([L,A])$ for any Lie ideal $L$; this is a consequence of Herstein's \cite[Theorem 1]{herstein} applied in $A/\Id([L,L])$.) Thus, there exist $x_1,y_1,\cdots,x_m,y_m\in A$ and $p_1,q_1,\cdots,p_m,q_m\in P$ such that
\[
1=\sum_{i=1}^m x_i[p_i,q_i]y_i.
\] 
Let $c\in A$ be a contraction. Then
\[
[c^*,c]=\sum_{i=1}^m[c^*x_i[p_i,q_i]y_i,c].
\]
By Lemma \ref{lieidentities} (i), the right side is a sum of terms of the form $[p_i,r]$, $[q_i,r]$, $[[p_i,r],s]$ or $[[q_i,r],s]$.
More explicitely,
\begin{equation}\label{morexplicit2}
[c^*,c]=\sum_{i=1}^m[p_i,r_i^{(1)}]+\sum_{i=1}^m[q_i,r_i^{(2)}]+
\sum_{i=1}^m\sum_{k=1}^{M_{i}}[[p_i,r_k^{(3)}],s_k^{(1)}]+
\sum_{i=1}^m\sum_{k=1}^{M_{i}}[[q_i,r_k^{(4)}],s_k^{(2)}].
\end{equation}
 Let us focus on one of the terms of the form 
$[p,r]$ where $p=p_i$ and $r=r_i^{(j)}$ for some $i$ and $j$. Since $r$ is a polynomial in $c$ and $c^*$ with coefficients in $A$,  there exists $M>0$
such that $\|r\| \leq M$ for all  contractions $c$. Enlarging the number of terms $[p,r]$ if necessary
(by an amount depending on $M$ but independent of $c$)  let us assume that $\|r\|\leq 1$. Since $[c^*,c]$
is selfadjoint, we can replace  $[p,r]$, on the right side of \eqref{morexplicit2}, by its selfadjoint part. Thus, we may assume that
$r$ is skewadjoint.  
	For each $x\in pA(1-p)$ such that  $\|x\|\leq 1$  let us define
	\[
	q(x)=
	\begin{pmatrix}
	\frac{1+\sqrt{1-xx^*}}{2} & \frac x 2\\
	\frac{x^*}{2} & \frac{1-\sqrt{1-xx^*}}{2}
	\end{pmatrix}
	\in 
	\begin{pmatrix}
	pAp & pA(1-p)\\
	(1-p)Ap & (1-p)A(1-p)
	\end{pmatrix}.
	\] 
	A straightforward computation shows that   $q(x)$ is a projection.
Now let $x=pr(1-p)$.  Then	
\[
[p,r]=\begin{pmatrix}
	0 & x\\
	x^* & 0
	\end{pmatrix}=q(x)-q(-x). 
\]
Let us consider the terms of the form $[[p,r],s]$. Enlarging the number of terms again we may assume that 
$\|r\|\leq 1$ and $\|s\|\leq 1$. Braking up $r$ and $s$ into the sum of a selfadjoint and a skewadjoint and taking the selfadjoint part
of  $[[p,r],s]$ we can assume that either both $r$ and $s$ are selfadjoint or both are skewadjoint. The first case reduces
to the second by setting $[[p,r],s]]=[[p,ir],-is]$. In the case that both $r$ and $s$ are skewadjoint we apply twice the argument used above for  terms  $[p,r]$  with $r$ skewadjoint. We thus find that the selfadjoint parts of $[p,r]$ and $[[p,r],s]$ are expressible as sums
whose terms have the form $\epsilon q$, with $\epsilon=\pm 1$ and $q$ a projection. Applying this to each term on the right side of \eqref{morexplicit2} proves the theorem. 
\end{proof}

\section{Non-closed normal subgroups of invertibles}\label{nonclosedGA}
As before, $A$ denotes a $\Cstar$-algebra.
In this section we prove results on normal subgroups of $\G_A$ that are not assumed to be closed at the outset (and may well fail to be so).  Our strategy is  again  to relate normal subgroups to Lie ideals, which may not be closed either. 
We also look closer into the properties of the exponential map in a small neighborhood of the identity.
Theorem \ref{KVeasy} below is a simple consequence of the Campbell-Baker-Hausdorff formula.
Theorem \ref{KVhard} follows from the formalism associated to the first Kashiwara-Vergne equation, as developed by Rouvi\`{e}re in \cite{rouviere}. Both theorems will be  proven in the appendix.

\begin{theorem}\label{KVeasy}
Let  $x,y\in A$ be such that $\|x\|,\|y\|<\frac{\log 2}{8}$. Then
	there exist elements $a=a(x,y)\in A$ and $b=b(x,y)\in A$, depending continuously on $x$ and $y$,  such that 
		\[
e^{x}e^y=e^{x+y+[x,a]+[y,b]},		
\]
and $a(0,0)=b(0,0)=0$.
Moreover, if $x$ and $y$    are skewadjoint then so are  $a$ and $b$.
\end{theorem}
\begin{proof}
See Theorem \ref{KVeasyApp} of the Appendix.
\end{proof}

\begin{theorem}\label{KVhard}
There exists $\epsilon>0$ such that if  $x,y\in A$ are such that
		$\|x\|,\|y\|\leq \epsilon$ then
there exist $R=R(x,y)\in A$ and $S=S(x,y)\in A$, depending continuously on $x$ and $y$,  such that
		\[
		e^{x+y}=(e^Re^xe^{-R})\cdot (e^Se^ye^{-S}),
		\]
		and 
		\begin{align*}
		R=\frac{y}{4}+[x,R']+[y,R''],\\
		S=-\frac{x}{4}+[x,S']+[y,S''],
		\end{align*}
		for some $R',R'',S',S''\in A$. Moreover,  if $x$ and $y$ are skewadjoint then so are $R$ and $S$.
	\end{theorem}
\begin{proof}
See Theorem \ref{fgODE} of the Appendix.
\end{proof}

\begin{lemma}\label{break}
Let $A$ be a $\Cstar$-algebra. Let $b_1,\dots,b_n\in A$. Then $e^{b_1}\cdots e^{b_n}$
and $e^{\sum_{i=1}^n b_i}$ are equal modulo $(\G_A,\G_A)$.
\end{lemma}
\begin{proof}
For elements $g,h\in \G_A$ let us write $g\sim h$ if they are equal modulo $(\G_A,\G_A)$; i.e.,$gh^{-1}\in (\G_A,\G_A)$. It suffices to show that $e^{b_1}e^{b_2}\sim e^{b_1+b_2}$.
Let $N\in \N$. Let us choose $N$ large enough such that, by Theorem \ref{KVhard},
we have that $e^{\frac{b_1}N+\frac{b_2}N}=e^Re^{\frac{b_1}N}e^{-R}e^Se^{\frac{b_2}N}e^{-S}$ for some $R,S$. Then
\[
e^{b_1+b_1}=(e^{\frac{b_1}N+\frac{b_2}N})^N=
(e^Re^{\frac{b_1}N}e^{-R}e^Se^{\frac{b_2}N}e^{-S})^N\sim e^{b_1}e^{b_2}.\qedhere
\]
\end{proof}

\begin{theorem}\label{GAGA}
Let $A$ be a   $\Cstar$-algebra.
The following describe the group $(\G_A,\G_A)$:
\begin{enumerate}[(a)]
\item
the group generated by the set $\{e^{[c,d]}\mid c,d\in A\}$,
\item
the set of elements of the form 
$
e^{b_1}e^{b_2}\cdots e^{b_n}
$
where $\sum_{i=1}^n b_i\in [A,A]$.
\end{enumerate}
If $A$ is unital and has no 1-dimensional representations then $(\G_A,\G_A)$
is also equal to
\begin{enumerate}
\item[(c)] the group generated by $1+N_2$.
\end{enumerate}
Moreover, in this case  $(\G_A,\G_A)$ is a perfect group.
\end{theorem}	
\begin{proof}
It is clear that the group described in (a) is contained in the group described in (b). Let us prove that the latter is contained in $(\G_A,\G_A)$. By Lemma \ref{break}, 
proving that $\prod_{i=1}^n e^{b_i}\in (\G_A,\G_A)$ for $\sum_{i=1} b_i\in [A,A]$
reduces to proving that $e^b\in (\G_A,\G_A)$ for  $b\in [A,A]$.
Observe then that
\[
[A,A]=[\G_A,A]=\Lin(\{hah^{-1}-a\mid h\in \G_A,a\in A\}).
\]
Thus, we may express $b$ as  $\sum_{i=1}^m h_ia_ih_i^{-1}-a_i$
for some $h_i\in \G_A$ and $a_i\in A$. By Lemma \ref{break} again we get that
\[
e^b\sim \prod_{i=1}^m e^{h_iah_i^{-1}}e^{-a_i}=\prod_{i=1}^m(h_i,e^{a_i})\in (\G_A,\G_A).
\]
Next let us show that $(\G_A,\G_A)$ is contained in the group described in (a).
Call the latter group $H$. Notice that $H$ is normal in $\G_A$.
Thus, as argued above in similar situations,  to prove that $(\G_A,\G_A)\subseteq H$ it suffices to prove that $(b,e^a)\in H$ for all $b\in \G_A$ and $a\in A$, and  we can assume that $a$ is  small (how small to be specified below). 
By Theorem \ref{KVeasy} for $a$ small enough we have that 
	\begin{align*}
	(b,e^{a})  &= 
e^{bab^{-1}}e^{-a} \\
&=e^{bab^{-1}-a+[bab^{-1},x]-[a,y]},
	\end{align*}
for some $x,y\in A$. Observe that $bab^{-1}-a+[bab^{-1},x]-[a,y]\in [A,A]$. Thus, it suffices to show that if $c\in [A,A]$ then $e^c\in H$. Say $c=\sum_{i=1}^m [e_i,f_i]$. We consider $e^{\frac c N}$ and choose $N$ large enough such that, by Theorem \ref{KVhard}, 
\[
e^{\frac c N}=e^Re^{[e_1,\frac {f_1}N]}e^{-R}e^{S}
e^{\sum_{i=2}^m[e_i,\frac{f_i}N]}e^{-S}.
\] 
The first factor $e^Re^{[e_1,\frac {f_1}N]}e^{-R}$ is in $H$, so we can continue the argument inductively. We thus find that $e^{\frac c N}\in H$ for large enough $N$, which in turn implies that $e^c\in H$, as desired.

Finally, suppose that $A$ is unital and has no 1-dimensional representations.
Let $U_2$ denote the group generated by $1+N_2$. Observe that $U_2$ is normal.
We already know, by Lemma \ref{N2comm}, that   $U_2\subseteq ((\G_A,\G_A),(\G_A,\G_A))$. Thus, once we have shown that $(\G_A,\G_A)\subseteq U_2$ it will both follow that 
$U_2=(\G_A,\G_A)$  and that $(\G_A,\G_A)$  is perfect. 
Let us show that $e^{[c,d]}\in U_2$ for all $c,d\in A$. 
	 By Theorem \ref{commN2}, there exist $y_1,\dots,y_K\in N_2$ such that 
$[c,d]=\sum_{i=1}^K y_i$. Let us show that the exponential of a finite sum of elements
in $N_2$ belongs to $U_2$. This is certainly true for sums of one term: $e^y=1+y\in U_2$ for $y\in N_2$.
 Suppose it is true for sums of $K-1$ terms. Let $y_1,\dots,y_K\in N_2$.
Let $N\in \N$. Let us choose $N$
large enough such that Theorem \ref{KVhard} can be applied to  
$e^{\frac{y_1}{N}+\sum_{i=2}^K\frac {y_k}N}$.
Then
\[
e^{\frac{1}{N}\sum_{i=1}^K y_i}=
e^Re^{\frac{y_1}{N}}e^{-R}e^Se^{\sum_{i=2}^K\frac{y_i}{N}}e^{-S}.
\]
The right side belongs to $U_2$ by induction. Thus, raising the left side to the $N$
we find that it belongs to $U_2$, as desired.
\end{proof}

\begin{proposition}\label{moreconcrete}
Let $A$ be a  $\Cstar$-algebra. Let  $H$ be a subgroup of $\G_A$ containing
$(\G_A,\G_A)$. Then there exists  an additive subgroup   $L\subseteq A$ containing
$[A,A]$  such that 
\[
H=\{e^{b_1}\cdots e^{b_n}\mid \sum_{i=1}^n b_i\in L\}.
\]
\end{proposition}	
\begin{proof}
Define $L\subseteq A$	as 
\[
L=\Big\{\sum_{i=1}^n b_i\mid \prod_{i=1}^n e^{b_i}\in H\Big\}.
\]
It is clear that $L$ is an additive subgroup of $A$ (since $H$ is a group). 
From $(\G_A,\G_A)\subseteq H$  and Theorem \ref{GAGA}  we deduce that $[A,A]\subseteq L$. The inclusion of $H$
in $\{e^{b_1}\cdots e^{b_n}\mid \sum_{i=1}^n b_i\in L\}$ is clear from the  definition of $L$. 
Next we prove the opposite inclusion.

Consider a product $\prod_{i=1}^n e^{b_i}$ such that $b:=\sum_{i=1}^n b_i\in L$. By 
Lemma \ref{break}, $\prod_{i=1}^n e^{b_i}\sim e^b$ (equal modulo commutators).
Since $H$ contains $(\G_A,\G_A)$, $\prod_{i=1}^n e^{b_i}$ is in $H$ if and only if $e^b\in H$. Since $b\in L$, there must exist $c_1,\dots,c_m\in A$ such that $e^{c_1}\cdots e^{c_m}\in H$
and $b=\sum_{j=1}^m c_j$. By Lemma \ref{break}, $e^b\sim e^{c_1}\cdots e^{c_m}\in H$, which implies that $e^{b}\in H$, as desired.
\end{proof}

We now  prove the main result of the section.
 
\begin{theorem}\label{HGAGA}
	Let $A$ be a unital $\Cstar$-algebra. Let $H$ be a subgroup of $\G_A$ normalized by $(\G_A,\G_A)$ and such that $A=\Id([H,A])$.
	Then $(\G_A,\G_A)\subseteq H$. 
\end{theorem}	

\begin{proof}
To prove that $(\G_A,\G_A)\subseteq H$ it suffices, by Theorem \ref{GAGA},  to show  that $e^{[c,d]}\in H$ for all $c,d\in A$. 
Before doing so we establish some preparatory results.

Let	
\[
\Lambda=\{hah^{-1}-a \mid h\in H,\,a\in [A,A]\}. 
\]
Our first goal is to show that  $\Id(\pi_3(\Lambda))=A$, where $\pi_3$ is the polynomial defined in \eqref{pin}. 
(Here and below $\pi_3(S,\dots,S)$ is abbreviated  to $\pi_3(S)$ for various sets $S\subseteq A$.)
To prove this, consider first the ideal $\Id([\Lambda,A])$. Let $\widetilde A$ denote the quotient of $A$ by this ideal and let $\widetilde H$ denote the image of $H$ in this quotient. Let $h\in \widetilde H$ and $a\in [\widetilde A,\widetilde A]$. Then $hah^{-1}-a$ is a central element in $\widetilde A$. In particular, it commutes with $h$. Hence, $(hah^{-1}-a)h=[h,a]$ commutes with $h$; i.e., $[h,[h,a]]=0$ for all $a\in [\widetilde A,\widetilde A]$. By Herstein's theorem, $[h,a]=0$ for all $a\in [\widetilde A,\widetilde A]$. In particular,
$[h,[h,b]]=0$ for all $b\in \tilde A$. By Herstein's theorem again, $[h,b]=0$. 
Hence, $[H,A]\subseteq \Id([\Lambda,A])$, and since we have assumed that $[H,A]$ is full,
$\Id([\Lambda,A])=A$. Next  notice that $\overline{\Lin(\Lambda)}$ is a closed subspace  invariant under conjugation by $(\G_A,\G_A)$. Hence, by \cite[Theorem 2.6]{lie}, it  is a Lie ideal of $A$. From the definition of $\Lambda$ we see clearly that $\overline{\Lin(\Lambda)}$ is contained in $\overline{[A,A]}$. Moreover,
$\Id([\Lin(\Lambda),A])=\Id([\Lambda,A])=A$.  It follows by  \cite[Lemma 1.6]{lie} that   $\overline{\Lin(\Lambda)}=\overline{[A,A]}$. Since $\pi_3$ is a multilinear polynomial, we have
$\Id(\pi_3(\Lambda))=\Id(\pi_3([A,A]))$.
 But
 $\overline{\pi_3([A,A])}=\overline{[A,A]}$, by  \cite[Corollary 2.8]{lie}, and $\Id([A,A])=A$  by assumption. 
Hence, $\Id(\pi_3(\Lambda))=A$, as desired.

Since $A$ is unital and $\pi_3(\Lambda)$ is a full set
there exist $\overline x_1,\dots,\overline x_m\in \Lambda^8$ and $y_1,z_1,\dots,y_m,z_m\in A$ such that
\[
1=\sum_{i=1}^m y_i\pi_{3}(\overline x_i)z_i.
\]
For each $i=1,\dots,m$ let's say that 
\[
\overline x_{i}=(h_{i,1}a_{i,1}h_{i,1}^{-1}-a_{i,1},\dots, h_{i,1}a_{i,8}h_{i,8}^{-1}-a_{i,8}),
\] 
where $h_{i,1},\dots,h_{i,8}\in H$ and $a_{i,1},\dots,a_{i,8}\in [A,A]$.

Let  $h\in H$ and $a\in [A,A]$. Let $t\in \R$ be nonzero. Let us define 
\[
W_t(h,a)=\frac{1}{t}\log((h,e^{ta})).
\]
Since  $(h,e^{ta})\to 1$ as $t\to 0$, $W_t(h,a)$ is defined for small enough  $t$. 
 Notice that 
 \[
 \lim_{t\to 0} W_t(h,a)=hah^{-1}-a=:W_0(h,a).
 \]
Hence, if we set 
\[
\overline x_{i}(t)=(W_t(h_{i,1},a_{i,1}),\dots,W_t(h_{i,8},a_{i,8}))
\] 
and
\[
y(t)=\sum_{i=1}^m y_i\pi_{3}(\overline x_{i}(t))z_i,
\]
then $y(t)\to 1$ as $t\to 0$. In particular, there exists $\delta>0$ such that 
$y(t)$ is invertible for all $|t|\leq \delta$. We thus have that
\begin{equation}\label{1sum}
1=\sum_{i=1}^m y(t)^{-1}y_i\pi_{3}(\overline x_i(t))z_i,
\end{equation}
for all $|t|\leq \delta$.

Let  $c,d\in A$. Multiplying by $c$   and taking commutator with $d$ in \eqref{1sum} we get that
\begin{equation}\label{1sum1}
[c,d]=\sum_{i=1}^m [y_i(t)\pi_{3}(\overline x_i(t))z_i,d],
\end{equation}
for all $|t|\leq \delta$, where we have set $y_i(t)=cy(t)^{-1}y_i$.
By Lemma \ref{lieidentities}, for each $i=1,\dots,m$   the term  $[y_i(t)\pi_{3}(\overline x_i(t))z_i,d]$
is expressible as a
sum whose terms have either the form $[W_t(h,a),[r,s]]$ or 
the form $[[W_t(h,a),[r,s]],[r',s']]$. That is,
\begin{equation*}
[c,d]=\sum_{i,l}\sum_{r,s} [W_t(h_{i,l},a_{i,l}),[r,s]]+\sum_{i,l}\sum_{r,s,r',s'} [[W_t(h_{i,l},a_{i,l}),[r,s]],[r',s']] 
\end{equation*}
Here $i=1,\dots,m$, $l=1,\dots,8$, and for each $i,l$ the inner sums run through  finitely many polynomials 
$r,s,r',s'$ on
$y_i(t), W_t(h_{i,l},a_{i,l}),z_i,d$. Note that  $r,s,r',s'$  are continuous functions of  $t$. 
Let us consider first the terms of the form $[W_t(h,a),[r,s]]$. 
By Theorem \ref{commN2}, we have
$[r,s]=\sum_{i=1}^K y_i$ 
for some $y_1,\dots,y_K\in N_2$ such that $\|y_i\|\leq C\|r\|\cdot \|s\|$ for all $i$.  So 
\begin{equation}\label{1sum2}
[W_t(h,a),[r,s]]=\sum_{i=1}^K [W_t(h,a),y_i].
\end{equation}
Let us abbreviate $fzf^{-1}$ to $f\cdot z$. On each term $[W_t(h,a),y_i]$ on the right of \eqref{1sum2} let us use that
\begin{align}
\nonumber
[W_t(h,a),y_i] 
&=(1+\frac {y_i} 2)\cdot W_t(h,a)+(1-\frac{y_i} 2)\cdot W_t(h,a)\\
&=W_t\Big((1+\frac {y_i} 2)\cdot h,(1+\frac {y_i} 2)\cdot a\Big)+
     W_t\Big((1-\frac {y_i} 2)\cdot h,(1-\frac {y_i} 2)\cdot a\Big).\label{N2commsim}
\end{align}
Notice that $(1+z)h(1-z)\in H$ for all $z\in N_2$ since $H$ is normalized by $(\G_A,\G_A)$.  So in this way we can express an element of the form $[W_t(h,a),[r,s]]$ as a sum whose terms have the form $W_t(h',a')$, where $h'\in H$
and $a'\in [A,A]$ are similarity conjugates of $h$ and $a$ respectively.
The same argument  used twice yields that elements  of the form 
$[[W_t(h,a),[r,s]],[r',s']]$ are expressible as sums whose  terms have the form
$W_t(h'',a'')$, where $h''\in H$ and $a''\in [A,A]$ are similarity conjugates of $h$ and $a$. In summary, we have shown that
\begin{equation}\label{2sum}
[c,d]=\sum_{j=1}^M W_t(h_j',a_j'),
\end{equation}
for some $h_1',\dots,h_M'\in H$ and $a_1',\dots,a_M'\in [A,A]$ which depend on $t$. We claim however that $M$ is independent of $t$ and that the norms of the elements $h_j'$ and $a_j'$ are uniformly bounded for all $|t|\leq \delta$. The number of terms in \eqref{2sum} resulted from starting with \eqref{1sum1}, expanding each of its terms first using Lemma \ref{lieidentities},  
then using Theorem \ref{commN2} on the new terms to get \eqref{1sum2}, and further applying \eqref{N2commsim}. These last two steps were repeated on terms of the form $[[W_t(h,a),[r,s]],[r',s']]$.  This makes it clear that $M$ is independent from $t$.
Since the elements $r,s,r',s'$ depend continuosly on $t$ their norms are uniformly bounded on $|t|\leq \delta$. Then Theorem \ref{commN2} and inspection of \eqref{N2commsim} guarantee the uniform boundedness 
of the final $h_j'$s and $a_j'$s on 
$|t|\leq \delta$.

Let's multiply by $t$ on both sides of \eqref{2sum} and then exponentiate:
\begin{equation}\label{almostdone}
e^{t[c,d]}=e^{\sum_{j=1}^M tW_t(h_j',a_j')}.
\end{equation}
Let us show that the right side belongs to $H$ for small enough $t$. Indeed, for small enough $t$  Theorem \ref{KVhard}  is applicable on the right side because the elements
$W_t(h_j',a_j')$ are uniformly bounded on $|t|\leq \delta$. We then get that
\begin{align}\label{almostalmostdone}
e^{\sum_{j=1}^M tW_t(h_j',a_j')} &=(e^R e^{tW_t(h_1',a_1')} e^{-R})
\cdot (e^S e^{\sum_{j=2}^M tW_t(h_j',a_j')}e^{-S}).
\end{align}
Observe that, by Theorem \ref{KVeasy}, 
\[
W_t(h_j',a_j')=h_j'a_j'(h_j')^{-1}-a_j'+[\cdot,\cdot]+[\cdot,\cdot]\in [A,A]
\] 
for all $j$. Then Theorem \ref{KVhard} guarantees that $R,S\in [A,A]$.
It is also true that $a_j'\in [A,A]$ for all $j$, since each $a_j'$ is conjugate to some $a_{i,l}\in [A,A]$. Hence, by Theorem\ref{GAGA}, $e^R,e^S\in (\G_A,\G_A)$ and $e^{ta_j'}\in (\G_A,\G_A)$ for all $j$ and $t$. Since  we have assumed that $H$
is normalized by $(\G_A,\G_A)$,
\[
e^Re^{tW_t(h_1',a_1')}e^{-R}=e^R(h_1,e^{ta_1'})e^{-R}\in H.
\]
That is, the first factor on the right of \eqref{almostalmostdone} is in $H$. Continuing the same argument inductively 
we conclude that the right side of \eqref{almostdone} belongs to $H$ for small enough $t$. Hence,  $e^{t[c,d]}$ belongs to $H$ for  small enough $t$. Since $e^{[c,d]}=(e^{[c,d]/n})^n$ for all $n\in \N$, $e^{[c,d]}$ belongs to $H$ for all $c,d\in A$, as desired.
\end{proof}

\begin{proof}[Proof of Theorem \ref{HGAGAintro} from the introduction]
This is simply Proposition \ref{moreconcrete} and Theorem \ref{GAGA} put together. 
\end{proof}

For  a simple unital $A$ we obtain as a corollary 
a multiplicative analogue of  \cite{hersteinbook}*{Theorem 1.12}.

\begin{corollary}
	Let $A$ be a simple unital $\Cstar$-algebra. Let $H$ be a subgroup of $\G_A$ normalized by $(\G_A,\G_A)$. Then either $H$ is contained in the center of $A$ or $(\G_A,\G_A)\subseteq H$.
The group $(\G_A,\G_A)$ is simple modulo its center.		
\end{corollary}	
\begin{proof}
Suppose that $H$ is not contained in the center of $A$. Then $\Id([H,A])\neq \{0\}$ and so $\Id([H,A])=A$
by the simplicity of $A$. By the previous theorem, $(\G_A,\G_A)=(H,\G_A)\subseteq H$.
\end{proof}

Let $M_\infty(A)=\bigcup_{n=1}^\infty M_n(A)$, where each matrix algebra is embedded in the next as a corner algebra by the map  $a\mapsto \left(\begin{smallmatrix} a&0\\0&0 \end{smallmatrix} \right)$. 
Let $\G_A^{(n)}=\G_{M_n(A)}$ for all $n=1,2\dots$ and $\G_A^{(\infty)}=\bigcup_{n=1}^\infty \G_A^{(n)}$, where each group is embedded in the next by the map
 $g\mapsto \left(\begin{smallmatrix} g&0\\0&1\end{smallmatrix}\right)$.
In \cite{dlHarpe-Skandalis1} de la Harpe and Skandalis define a determinant map on $\G_A^{(\infty)}$ associated to 
the ``universal trace" $A\mapsto A/\overline{[A,A]}$. 
Let us recall a description of the kernel of the  de la Harpe-Skandalis determinant (we do not need
recalling the definition of the determinant itself): 
An element $a\in \G_A^{(\infty)}$ is in the kernel of the de la Harpe-Skandalis determinant if   
whenever $a$ is written as a product $e^{b_1}\cdots e^{b_n}$, with $b_1,\dots,b_n\in M_\infty(A)$, then
\[
\sum_{i=1}^n b_i\in \overline{[A,A]}+2\pi i\{\mathrm{Tr}(p-q)\mid,p,q \hbox{ projections in }M_\infty(A)\}.
\] 
Here  $\mathrm{Tr}((a_{i,j})_{i,j=1}^n):=\sum_{i=1}^n a_{i,i}$.  
Comparing the description of the kernel of the determinant map with the description of $(\G_A,\G_A)$ obtained in Theorem (iii) the following theorem becomes rather clear:
\begin{theorem}
	Let $A$ be a   C*-algebra such  that
\begin{enumerate}
\item[(a)]
$\overline{[A,A]}=[A,A]$.
\end{enumerate}
Then the kernel of the de la Harpe-Skandalis determinant agrees with $(\G_A^{(\infty)},\G_A^{(\infty)})$.
Suppose further that 
\begin{enumerate}
\item[(b)]
for each projection  $p\in M_\infty(A)$ there exists $h\in A_{\sa}$ such that $\mathrm{Tr}(p)-h\in [A,A]$ and $e^{2\pi ih}\in (\G_A,\G_A)$.
\end{enumerate}
Then the kernel of the de la Harpe-Skandalis determinant restricted to $\G_A$ agrees with $(\G_A,\G_A)$.
\end{theorem}	

\begin{proof}
That  $(\G_A^{(\infty)},\G_A^{(\infty)})$ is contained in the kernel of the determinant map 
is proven in \cite{dlHarpe-Skandalis1}.  Let $a=e^{b_1}\cdots e^{b_n}\in \G_A^{(\infty)}$ be in the kernel of the de la Harpe-Skandalis determinant. Then $b_1,\dots,b_n\in M_\infty(A)$ are such  
$\sum_{i=1}^n \mathrm{Tr}(b_i)+2\pi i\mathrm{Tr}(p-q)\in \overline{[A,A]}=[A,A]$ for some projections $p,q\in M_\infty(A)$. The  projections $p,q$ can be added to the list $b_1,\dots,b_n$ so that we may assume that $\sum_{i=1}^n \mathrm{Tr}(b_i)\in [A,A]$ (using that $e^{2\pi i p}=e^{2\pi i q}=1$).   Let us choose $N\geq 1$ such that $b_1,\dots,b_n\in M_N(A)$. 
 Then  $\sum_{i=1}^n \mathrm{Tr}(b_i)\in [A,A]$ is known to imply that
$b_1+\cdots +b_n\in [M_N(A),M_N(A)]$ (see \cite[Lemmas 2.1 and 2.2]{marcoux06}).   Hence, $a\in (\G_{A}^{(N)},\G_{A}^{(N)})$  by Theorem \ref{GAGA}.

Suppose that $A$ also satisfies (b). Let $a\in \G_A$ be in the kernel of the determinant map. Let us write $a=e^{b_1}\cdots e^{b_n}$, with $b_1,\dots,b_n\in A$. Let $p,q\in M_\infty(A)$
be projections such that $\sum_{i=1}^n b_i-2\pi i\mathrm{Tr}(p-q)\in [A,A]$. By assumption, we can  choose  $c,d\in A_{\sa}$  such that $\mathrm{Tr}(p)-c$ and  $\mathrm{Tr}(q)-d$ are in  $[A,A]$ and $e^{2\pi i c}$ and $e^{2\pi id}$ in $(\G_A,\G_A)$.
Then
\[
a=(e^{b_1}\cdots e^{b_n} e^{-2\pi ic}e^{2\pi id})\cdot (e^{-2\pi id})\cdot (e^{2\pi ic}).
\] 
The three factors on the right side are in $(\G_A,\G_A)$ either by assumption or by Theorem \ref{GAGA}.
\end{proof}	
	
A $\Cstar$-algebra is called pure if its Cuntz semigroup is both almost divisible and almost unperforated. 	
By results of \cite{tracedim}  the hypotheses of the theorem above are satisfied if $A$ is a unital pure $\Cstar$-algebra whose  bounded 2-quasitraces are traces. That $[A,A]=\overline{[A,A]}$ in this case is shown in \cite[Theorem 4.10]{tracedim}. That for each projection $p\in M_\infty(A)$ there exists $h\in A_{\sa}$ such that $\mathrm{Tr}(p)-h\in [A,A]$ and $e^{2\pi ih}\in (\G_A,\G_A)$ is shown in \cite[Lemma 6.2]{tracedim}. We thus obtain the following corollary:

\begin{corollary}
	Let $A$ be  unital pure $\Cstar$-algebra whose bounded 2-quastraces are traces. Then the kernel of the de la Harpe-Skandalis determinant on $\G_A$ agrees with $(\G_A,\G_A)$.
\end{corollary}

\section{Non-closed normal subgroups of unitaries}\label{nonclosedUA}
In this section we prove results on non-closed normal subgroups of $\U_A$.
We follow the same general strategy from the last section, so in some passages our arguments will  be more  succint. 
 	
\begin{lemma}\label{unitarybreak}
Let $A$ be a C*-algebra and $b_1,\dots,b_n\in A_{\sa}$. Then $e^{ib_1}\cdots e^{ib_n}$
is equal to $e^{i\sum_{i=1}^n b_i}$ modulo $(\U_A,\U_A)$.
\end{lemma}
\begin{proof}
The same proof used in Lemma \ref{break} applies here. It must be kept in mind that in
Theorem \ref{KVhard} if the elements $x$ and $y$ are skewadjoint then $R$ and $S$
also are, and so $e^R$ and $e^S$ are unitaries.
\end{proof}

\begin{theorem}
Let $A$ be a $\Cstar$-algebra. The following describe the group $(\U_A,\U_A)$:
\begin{enumerate}[(a)]
\item
the group generated by $\{e^{i[c^*,c]}\mid c\in A\}$,

\item
the elements of the form $e^{ib_1}\cdot \cdots\cdot e^{ib_n}$
with $\sum_{i=1}^n b_i\in [A,A]$ and $b_i\in A_{\sa}$ for all $i$.
\end{enumerate}
If $A$ is unital  and without 1-dimensional representations then $(\U_A,\U_A)$ is also equal to
\begin{enumerate}
\item[(c)]
 the group generated by $\{e^{i[x^*,x]}\mid x\in N_2\}$.
\end{enumerate}
 Moreover, in this case the group $(\U_A,\U_A)$ is perfect. 
\end{theorem}
\begin{proof}
It is clear that the group described in (a) is contained in the one described in (b). Let us prove that the latter is contained in $(\U_A,\U_A)$. By Lemma \ref{unitarybreak} it suffices to show that 
$e^{ib}\in (\U_A,\U_A)$ for $b\in [A,A]\cap A_{\sa}$. Observe that
\[
[A,A]=[\U_A,A]=\{uau^*-a\mid u\in \U_A,a\in A\}.
\]
So if $b\in [A,A]\cap A_{\sa}$ then $b$ can be expressed as a sum $\sum_{j=1}^n u_ja_ju_j^*-a_i$ for some $u_j\in \U_A$ and $a_j\in A$. Taking the selfadjoint part of this sum
we may assume that $a_j\in A_{\sa}$ for all $j$. Then, by Lemm \ref{unitarybreak},
\[
e^{ib}\sim \prod_{j=1}^ne^{u_j(ia_j)u_j^*}e^{-ia_j}=\prod_{j=1}^n (u_j,e^{ia_j})\in (\U_A,\U_A).
\]

Let us show next that $(\U_A,\U_A)$ is contained in the group described in (a). Call this group $H$. Observe that $H$ is normal in $\U_A$. It thus suffices to show that $(u,e^{ia})\in H$ for all $u\in \U_A$
and $a\in A_{\sa}$, and $a$ can be chosen small. Let us choose $a$ small enough such that,
by Theorem \ref{KVeasy},
\[
(u,e^{ia})=e^{i(uau^*-a)+i[uau^*,x]-i[a,y]},
\]
for some $x$ and $y$ skewadjoint. It now suffices to show that if $c\in [A,A]\cap A_{\sa}$
then $e^{ic}\in H$. The selfadjoint part of a commutator $[e,f]$ can always be
expressed in the form $[z^*,z]+[(z')^*,z']$ (see \cite[Theorem 2.4]{marcoux06}). So $c=\sum_{j=1}^m [z_j^*,z_j]$ for some
$z_j\in A$. We consider $e^{i\frac c N}$ and choose $N$ large enough such that, by Theorem \ref{KVhard}, 
\[
e^{i\frac c N}=(e^Re^{\frac i N[z_j^*,z_j]} e^{-R})\cdot e^{S}
e^{\sum_{j=2}^m\frac i N[z_j^*,z_j]}e^{-S}.
\] 
Notice that Theorem \ref{KVhard} guarantees that $R,S$ are skewadjoint,  so  
$e^{R},e^{S}\in \U_A$. The  first factor on the right side   is in $H$, so we can continue arguing inductively  that $e^{i\frac c N}\in H$ for large enough $N$. This in turn implies that $e^{ic}\in H$, as desired.
 
Finally, suppose that $A$ is unital and has no 1-dimensional representations.
Let $U_2$ denote the group described in (c). Observe that $U_2$ is normal.
We already know, by Lemma \ref{harpeskand}, that   $U_2\subseteq ((\U_A,\U_A),(\U_A,\U_A))$. Thus, once we have shown that $(\U_A,\U_A)\subseteq U_2$ it will both follow that 
$U_2=(\U_A,\U_A)$  and that $(\U_A,\U_A)$  is perfect. 
Let us show that $e^{i[c^*,c]}\in U_2$ for all $c\in A$. 
	 By Theorem \ref{commN2}, there exist $y_1,\dots,y_K\in N_2$ such that 
$[c^*,c]=\sum_{j=1}^K y_j$. Since the left side is selfadjoint, we can replace the right side by its selfadjoint part. For $y\in N_2$ we have 
$\frac{y+y^*}{2}=[z^*,z]$
for some $z\in N_2$, by \cite[Lemma 5.1]{tracedim}. Thus, 
$[c^*,c]=\sum_{j=1}^K[z_j^*,z_j]$, where $z_j\in N_2$ for all $j$.
To prove that $e^{i\sum_{j=1}^K[z_j^*,z_j]}$ belongs to $U_2$ we again proceed inductively, relying on Theorem \ref{KVhard}: For large enough $N$ we have
\[
e^{\frac{1}{N}[c^*,c]}=
(e^Re^{\frac i N[z_1^*,z_1]}e^{-R})e^Se^{\sum_{j=2}^K \frac i N[z_j^*,z_j]}e^{-S}.
\]
The first factor on the right side belongs to $U_2$. Continuing by induction we get that $e^{\frac i N[c^*,c]}\in U_2$ for large enough $N$ which in turn implies that $e^{i[c^*,c]}\in U_2$, as desired.
\end{proof}

In the proof of Theorem \ref{HUAUA} below---the analogue of Theorem \ref{HGAGA}---we cannot make the arguments go through  without these additonal restrictions: we assume that the $\Cstar$-algebra 
has two full orthogonal projections and we deal only with normal subgroups of $\U_A$
rather than normalized by $(\U_A,\U_A)$. 

\begin{theorem}\label{HUAUA}
Let $A$ be a unital C*-algebra containing two full orthogonal projections. 
Let $H$ be a normal subgroup of $\U_A$ such that $\Id([H,\U_A])=A$. Then
$(\U_A,\U_A)\subseteq H$.
\end{theorem}
\begin{proof}
In view of the previous theorem, to prove that $(\U_A,\U_A)\subseteq H$  it suffices to show that $e^{i[c^*,c]}\in H$
for all $c\in A$ (or even in $c\in N_2$, but we will not need this).

Let $\Lambda=\{uau^*-a\mid u\in H,a\in  A_{\sa}\}$. We claim that $\Id(\pi_3(\Lambda))=A$.
To prove this, notice first that $\Lin(\Lambda)=[H,A]$. Since $\overline{\Lin(H)}$ is a closed subspace invariant under conjugation by $\U_A$, it is a Lie ideal (by \cite[Theorem 2.3]{marcoux-murphy}). Hence, by \cite[Theorem 5.27]{BKS} applied to $\overline{\Lin(H)}$, 
\[
\overline{[H,A]}=\overline{[\Id([H,A]),A]}=\overline{[[A,A],A]}=\overline{[A,A]}.
\]
So
$\overline{\Lin(\Lambda)}=\overline{[A,A]}$. From this, and continuing to argue as in the proof of Theorem \ref{HGAGA}, we conclude that $\Id(\pi_3(\Lambda))=A$, as claimed. 

Continuing with the argument of Theorem \ref{HGAGA}, we get that 
there exist $\overline x_1,\dots,\overline x_m\in \Lambda^8$ and $y_1,z_1,\dots,y_m,z_m\in A$ such that
\[
1=\sum_{i=1}^m y_i\pi_{3}(\overline x_i)z_i.
\]
Say  
\[
\overline x_{i}=(u_{i,1}a_{i,1}u_{i,1}^*-a_{i,1},\dots, u_{i,1}a_{i,8}u_{i,8}^*-a_{i,8}),
\] 
for all $i=1,\dots,m$, where $u_{i,1},\dots,u_{i,8}\in H$ and $a_{i,1},\dots,a_{i,8}\in  A_{\sa}$.

As in the proof of Theorem \ref{HGAGA}, 
let us define 
$W_t(u,a)=\frac{1}{it}\log((u,e^{ita}))$
for  $u\in H$, $a\in A_{\sa}$ and $t\in \R$ small enough and non-zero. Set $W_0(u,a)=uau^{*}-a$
so that $W_t(u,a)\to W_0(u,a)$ as $t\to 0$. 
Let
\[
\overline x_{i}(t)=(W_t(u_{i,1},a_{i,1}),\dots,W_t(u_{i,8},a_{i,8}))
\] 
and
\[
y(t)=\sum_{i=1}^m y_i\pi_{3}(\overline x_{i}(t))z_i.
\]
Then $y(t)\to 1$ as $t\to 0$. In particular, there exists $\delta>0$ such that 
$y(t)$ is invertible for $|t|\leq \delta$. We thus have that
\begin{equation*}
1=\sum_{i=1}^m y(t)^{-1}y_i\pi_{3}(\overline x_i(t))z_i,
\end{equation*}
for all $|t|\leq \delta$.

Let  $c\in A$. Then 
\begin{equation}\label{1sumunitary}
[c^*,c]=\sum_{i=1}^m [y_i(t)\pi_{3}(\overline x_i(t))z_i,c]
\end{equation}
for all $|t|\leq \delta$, where we have set $y_i(t)=c^*y(t)^{-1}y_i$.
By Lemma \ref{lieidentities}, each $[y_i(t)\pi_{3}(\overline x_i(t))z_i,c]$ on the right 
side is expressible as a
sum whose terms have either the form $[W_t(u,a),[r,s]]$ or the form
$[[W_t(u,a),[r,s]],[r',s']]$:
\begin{equation}\label{15sumunitary}
[c^*,c]=\sum_{i,l}\sum_{r,s} [W_t(u_{i,l},a_{i,l}),[r,s]]+\sum_{i,l}\sum_{r,s,r',s'} [[W_t(u_{i,l},a_{i,l}),[r,s]],[r',s']] 
\end{equation}
Here $i=1,\dots,m$, $l=1,\dots,8$, and for each $i,l$ the inner sums run through  a number of polynomials $r,s,r',s'$ on
$y_i(t), W_t(u_{i,l},a_{i,l}),z_i,d$. Notice that in particular these polynomials depend continuously on $t$. 
At this point  our arguments diverge from the proof of Theorem \ref{HGAGA}.
We do not wish to express $[W_t(u,a),[r,s]]$ 
as a sum of similarity conjugates of  $W_t(u,a)$, but rather, unitary conjugates. 
First, taking selfadjoint part in \eqref{15sumunitary}, and keeping in mind that $W_t(u,a)$ is selfadjoint,
we can replace $[r,s]$ in the terms of the form $[W_t(h,a),[r,s]]$ by its skewadjoint part. 
Say $r=r_1+ir_2$ and $s=s_1+is_2$, with $r_1,r_2,s_1,s_2\in A_{\sa}$. Then the skewadjoint
part of $[r,s]$ can be computed to equal
\[
\frac{i}{2}[(r_1+is_1)^*,r_1+is_1]+\frac{i}{2}[(r_2+is_2)^*,r_2+is_2]
\]
(see \cite[Theorem 2.4]{marcoux06}). Thus, the selfadjoint parts of the terms $[W_t(u,a),[r,s]]$ can be 
expanded into  sums of two terms of the form $[W_t(u,a),i[z^*,z]]$. 
Similarly, the  selfadjoint parts of elements of the form $[[W_t(u,a),[r,s],[r',s']]]$ can be expanded into  sums of four
elements of the form $[[W_t(u,a),i[z^*,z]],i[(z')^*,z']]$.
The elements $z$ and $z'$ depend continuously on $t$.
This implies that $\|z\|$ and $\|z'\|$ are uniformly bounded  for all $|t|\leq \delta$. Thus, by enlarging the number of terms
of the forms $[W_t(u,a),i[z^*,z]]$ and $[[W_t(u,a),i[z^*,z]],i[(z')^*,z']]$ if necessary we may assume that $\|z\|\leq 1$
and $\|z'\|\leq 1$. 
Then, by Theorem \ref{commP},  $\frac 1 2[z^*,z]$ is expressible as a sum of elements of the form $\pm p$,
where $p$ is a projection. Put differently, $[z^*,z]$ is a sum of elements of the form 
$(\pm 2)p$. So we get that
\begin{equation}\label{unasuma}
[W_t(u,a),i[z^*,z]]=\sum_{k=1}^K [W_t(h,a),2i\epsilon_kp_k],
\end{equation}
for some projections $p_1,\dots,p_K$ and $\epsilon_k=\pm 1 $ for all $k$.
Now we use that for any projection $p$ we have
\begin{equation}\label{unip}
[W_t(u,a),2ip]=v_pW_t(u,a)v_p^*-v_p^*W_t(u,a)v_p,
\end{equation}
where 
$v_p=p+i(1-p)$ is a unitary. Since  
 $vW(u,a)v^*=W(vuv^*,vav^*)$, we obtain  
 $[W_t(u,a),i[z^*,z]]$ expressed as a sum of elements of the form $W_t(u',a')$ where
 $u'$ and $a'$ are unitary conjugates of $u$ and $a$ (hence $u'\in H$).
The same argument applied twice
also implies that elements  of the form 
$[[W_t(u,a),i[z^*,z]],i[(z')^*,z']]$ are expressible as sums whose  terms have the form
$W_t(u'',a'')$, where $u''\in H$ and $a''\in A_{\sa}$ are unitary conjugates of $u$ and $a$. In summary, we have shown that
\begin{equation}\label{2sumunitary}
[c^*,c]=\sum_{j=1}^M W_t(u_j',a_j'),
\end{equation}
for some $u_1',\dots,u_M'\in H$ and $a_1',\dots,a_M'\in A_{\sa}$ which depend on $t$. We claim however that $M$ is independent of $t$ and that the norms of the elements $u_j'$ and $a_j'$ are uniformly bounded for all $|t|\leq \delta$. The number of terms in \eqref{2sumunitary} resulted as follows: starting from \eqref{1sumunitary} we expanded expanding each of its terms  using Lemma \ref{lieidentities}. The new terms were expanded further so as to have 
 terms of the forms $[W_t(u,a),i[z^*,z]]$ and $[[W_t(u,a),i[z^*,z]],i[(z')^*,z']]$. At this point we observed that the norms of $z$ and $z'$ were bounded by a constant independent of $t$. So these terms could be expanded further
 so as to have $\|z\|\leq 1$ and $\|z'\|\leq 1$. We 
then used Theorem \ref{commP} to get \eqref{unasuma}  and finally  \eqref{unip}. This makes it clear that $M$ is independent from $t$.
The elements $u_i'$ and $a_i'$ apperaing in \eqref{2sumunitary} are all unitary conjugates of the initial set of elements.
Hence their norms are independent of  $t$.

Let's multiply by $it$ on both sides of \eqref{2sumunitary} and then exponentiate:
\begin{equation}\label{almostdoneunitary}
e^{it[c^*,c]}=e^{\sum_{j=1}^M itW_t(u_j',a_j')}.
\end{equation}
Since the norms of $W_t(u_j',a_j')$ are bounded on $|t|\leq \delta$, for small enough $t$  Theorem \ref{KVhard} is applicable on the right side. We thus get that
\begin{align}\label{almostalmostdoneunitary}\nonumber
e^{\sum_{j=1}^M itW_t(u_j',a_j')} &=(e^R e^{itW_t(u_1',a_1')} e^{-R})
\cdot (e^S e^{\sum_{j=2}^M itW_t(u_j',a_j')}e^{-S})\\
&=(e^R(u_1',e^{ita_1'})e^{-R})\cdot (e^S e^{\sum_{j=2}^M itW_t(u_j',a_j')}e^{-S}).
\end{align}
Since  we have assumed that $H$ is normal,
the first factor on the right side of \eqref{almostalmostdoneunitary} is in $H$. Continuing the same argument inductively 
we conclude that the right side of \eqref{almostdoneunitary} belongs to $H$ for small enough $t$. Hence,  $e^{t[c^*,c]}$ belongs to $H$ for  small enough $t$ which in turn implies that $e^{[c^*,c]}$ belongs to $H$ for all $c\in A$, as desired.
\end{proof}

\begin{corollary}
Let $A$ be a simple unital $\Cstar$-algebra contaning a non-trivial projection. Let $H$ be a normal subgroup of $\U_A$.
Then either $H$ is contained in the center of $\U_A$ or $H$ contains $(\U_A,\U_A)$.
\end{corollary}

\section{Appendix}
Here we prove Theorems \ref{KVeasy} and \ref{KVhard}.

\subsection{Proof of Theorem \ref{KVeasy}}

In order to justify analytically some formal manipulations we sometimes work in the setting of 
a free Banach algebra. 
Let $E=\C\oplus \C$. Let us denote by $\X$ and $\Y$ the canonical basis of $E$.
Let $\A_n(\X,\Y)=E^{\otimes n}$ for $n\geq 1$. 
Let 
\begin{align*}
\A[\X,\Y] &=\bigoplus_{n=1}^\infty \A_n(\X,\Y),\\
\A[[\X,\Y]] &=\prod_{n=1}^\infty \A_n(\X,\Y),
\end{align*}
$\A[\X,\Y]$ is the free non-unital algebra in the variables $\X$ and $\Y$. We regard the elements of $\A[\X,\Y]$ as polynomials in the noncommuting variables $\X$ and $\Y$.
We regard the elements of $\A[[\X,\Y]]$ as formal power series in $\X$ and $\Y$.

Let us endow $E$ with the $\ell_1$ norm and then $\A_n(\X,\Y)=E^{\otimes n}$ with the projective tensor product norm (i.e., the $\ell_1$ norm
after identifying $E^{\otimes n}$ with $\C^{2^n}$).
Let $\A(\X,\Y)$ denote the Banach algebra of $\ell_1$ convergent series  in $\A[[\X,\Y]]$ endowed with the $\ell_1$ norm. We have the inclusions $\A[\X,\Y]\subseteq \A(\X,\Y)\subseteq \A[[\X,\Y]]$.   The Banach algebra $\A(\X,\Y)$ is the free Banach  algebra on the vector space $E$ (\cite{pestov}). 
Given $Z\in \A(\X,\Y)$ and elements in  $U,V$ in some Banach algebra $A$ and  of norm at most 1 we can evaluate $Z(U,V)$ via the universal property of $\A(\X,\Y)$, i.e., by extending the assignment $\X\mapsto U$, $\Y\mapsto V$ to a contractive algebra homomorphism.
The function $(U,V)\mapsto Z(U,V)$ is defined for $\|U\|\leq 1$ and  $\|V\|\leq 1$, and is given by a normally convergent power series in $U,V$.
In particular it is continuous and has Frechet derivatives of all orders for $\|U\|<1$ and $\|V\|<1$.

Let $t\in \R$.
 Define $\lambda_t\colon E\to E$ by $\lambda_t(v)=tv$ for all $v\in E$ and extend it to 
an algebra homomorphism on $\A[\X,\Y]$. Since $\lambda_t$ maps $\A_n[\X,\Y]$ to itself we can further extend it to $\A[[\X,\Y]]$. We call $\lambda_t$ the scaling automorphism. Notice that $\lambda_t$ maps the Banach algebra $\A(\X,\Y)$ to itself for $|t|\leq 1$.

Let $\Li_n(\X,\Y)\subseteq \A_n(\X,\Y)$ denote the span of the $n$-iterated commutators 
in $\X$ and $\Y$, i.e., $[v_1,[v_2,\cdots,[v_{n-1},v_n]\cdots ]]$ with $v_i\in E$.
Define also
\begin{align*}
\Li[X,Y] &=\bigoplus_{n=1}^\infty \Li_n(\X,Y),\\
\Li[[X,Y]]&=\prod_{n=1}^\infty\Li_n(\X,\Y).
\end{align*}
$\Li[X,Y]$ is  the free Lie algebra on $\X,\Y$.  
Finally, let $\Li(\X,\Y)$ denote  the closure of 
$\Li[X,Y]$ inside the Banach algebra $\A(\X,\Y)$.

Let $\nu_n\colon \A_n(\X,\Y)\to \Li_n(\X,\Y)$ denote the linear operator 
\[
\nu(v_{1}\cdots v_{n})=
\frac 1 n [v_1,[v_2,\cdots,[v_{n-1},v_n]\cdots]]
\] 
for $v_1,\dots,v_n\in E$. The Dynkin-Specht-Wever theorem asserts that  $\nu_n$ is the identity on $\Li_n(X,Y)$.
We can estimate that $\|\nu_n\|\leq 2^n$. From this we deduce that $\lambda_{\frac 1 2}\nu_n$ is a contractive map. 
 Define
$\nu\colon \A[\X,\Y]\to \A[\X,\Y]$ by $\nu=\sum_{n=1}^\infty \nu_n$. Observe  then $\lambda_{\frac 1 2}\nu$ extends to
a contractive map from $\A(\X,\Y)$ to $\Li(\X,\Y)$.

\begin{lemma}\label{half}
Let $Z=\sum_{n=1}^\infty Z_n\in \A(\X,\Y)$
be such that $Z_n\in \Li_n(\X,\Y)$ for all $n$ (i.e., $Z\in \Li[[\X,\Y]]$). 
Then $\lambda_{\frac 1 2}Z\in \Li(\X,\Y)$.
\end{lemma}

\begin{proof}
We have $\nu_n Z_n=Z_n$ for all $n$. So
$\lambda_{\frac 1 2}Z=\sum_{n=1}^\infty (\lambda_{\frac 1 2}\nu_n)Z_n\in \Li(\X,\Y)$
\end{proof}

\begin{lemma}\label{halfhalf}
Let $Z\in \Li[[\X,\Y]]$. Then there exist $P,Q\in \Li[[X,Y]]$ such that
\[
Z=Z_1+[\X,P]+[\Y,Q].
\] 
If $Z\in \Li(\X,\Y)$ then $\lambda_{\frac 1 2}P$ and $\lambda_{\frac 1 2}Q$ are in $\Li(X,Y)$.
\end{lemma}
\begin{proof}
Let $Z=\sum_{n=1}^\infty Z_n$ with $Z_n\in \Li_n(X,Y)$.
For each $n\geq 2$ we can write  $Z_n=\X P_{n}+\Y Q_{n}$, for some $P_n,Q_n\in \A_{n-1}(\X,\Y)$. We remark that, from the definition of the norm  on $\A_n(\X,\Y)$,
we have that $\|P_n\|\leq \|Z_{n-1}\|$ and $\|Q_n\|\leq \|Z_{n-1}\|$. Applying $\nu_n$ on both sides of $Z_n=\X P_{n}+\Y Q_{n}$ we get
$Z_n=[\X,\nu_n P_n]+[\Y,\nu_n Q_n]$, where now $\nu_nP_n$ and $\nu_nQ_n$ are in $\Li_{n-1}(\X,\Y)$. Let us set
$P=\sum_{n=2}^\infty \nu_n P_n$ and $Q=\sum_{n=2}^\infty\nu_n Q_n$.
Then $Z=Z_1+[\X,P]+[\Y,Q]$, as desired.
Since  $\lambda_{\frac 1 2}\nu_n$ is a contraction,  
$\|\lambda_{\frac 1 2} \nu_n P_n\|\leq \|Z_{n-1}\|$ and $\|\lambda_{\frac 1 2}\nu_n Q_n\|\leq \|Z_{n-1}\|$ for all $n\geq 2$. Hence, if $\sum_{n=1}^\infty \|Z_n\|<\infty$  then  $\lambda_{\frac 1 2}P,\lambda_{\frac 1 2}Q\in \Li(\X,\Y)$.
\end{proof}

Let $A$ be a Banach algebra and $X,Y\in A$. Define by holomorphic functional calculus
\[
V(X,Y)=\log(e^Xe^Y).
\]
$V(X,Y)$ is well defined for $\|X\|+\|Y\|<\log 2$.

\begin{theorem}\label{KVeasyApp}
Let $A$ be a Banach algebra. Let $0<\delta<\frac{\log 2}{8}$. There exist continuous functions $(X,Y)\mapsto A(X,Y)$ and $(X,Y)\mapsto B(X,Y)$ defined for $\|X\|,\|Y\|\leq \delta$ such that
\begin{equation}\label{Vanal}
V(X,Y)=X+Y+[X,A(X,Y)]+[Y,B(X,Y)]
\end{equation}
for all  $X,Y$. If $A$ is a $\Cstar$-algebra and $X$ and $Y$ are skewadjoint then so are 
$A(X,Y)$ and $B(X,Y)$
\end{theorem}

\begin{proof}
We first work in the Banach algebra $\A(\X,\Y)$. Since $\|4\delta \X\|+\|4\delta Y\|<\log 2$
we have 
$V(4\delta \X, 4\delta \Y)\in \A(\X,\Y)$. By the Campbell-Baker-Hausdorff Theorem, 
$V(4\delta \X, 4\delta \Y)\in \Li[[X,Y]]$. Thus, by Lemma  \ref{half}, $V(2\delta\X,2\delta\Y)\in \Li(\X,\Y)$.
Then, by Lemma \ref{halfhalf}, there exist  $P,Q\in \Li(\X,\Y)$ such that
\begin{equation}\label{Vformal}
V(\delta \X,\delta \Y)=\delta \X+\delta\Y+[\X,P]+[\Y,Q].
\end{equation}

Now let $A$ be a Banach algebra. For $X,Y\in A$ such that $\|X\|,\|Y\|\leq \delta $
define $A(X,Y)=P(\frac 1 \delta X,\frac 1 \delta Y)$ and $B(X,Y)=Q(\frac 1 \delta X,\frac 1 \delta Y)$. Applying the assignment  $\delta\X\mapsto X$, $\delta Y\mapsto Y$  in \eqref{Vformal} we get \eqref{Vanal}. The functions $(X,Y)\mapsto A(X,Y)$ and 
$(X,Y)\mapsto B(X,Y)$ are given by normally convergent series of iterated Lie brackets. In particular they are continuous.

Suppose now that $A$ is a $\Cstar$-algebra and that $X,Y$  are skewadjoint elements of $A$.
Observe that the involution $\sigma Z=-Z^*$ is a continuous Lie algebra homomorphism; i.e., $\sigma [Z_1,Z_2] =[\sigma Z_1,\sigma Z_2]$. Since $A(X,Y)$ is a convergent series of Lie brackets on $X,Y$ we get that $\sigma A(X,Y)=A(\sigma X,\sigma Y)=A(X,Y)$, i.e., $A(X,Y)$ is skewadjoint.  The same argument applies to $B(X,Y)$.
\end{proof}

\subsection{Proof of Theorem \ref{KVhard}}
Theorem \ref{KVhard} follows essentially from the collection formulas  related to the Kashiwara-Vergne equations, as developed in \cite[Sections 1.1--1.5]{rouviere}. We have pieced together the various parts of the argument from this reference (still  referring the reader  to \cite{rouviere} for some computations). Besides having tailored the statement of Theorem \ref{KVhard} to our purposes, our objective here has been to make it explicit  that
these formulas  are valid in the infinite dimensional setting of a Banach algebra (the setting in \cite{rouviere}  is either  formal or finite dimensional one).

Let $A$ be a Banach algebra. Let $X,Y\in A$ and $0<\delta<\frac{\log 2}{8}$. We have demonstrated in Theorem \ref{KVeasyApp} the existence of $A(X,Y)$
and $B(X,Y)$ for $\|X\|,\|Y\|\leq \delta$ such that
\begin{equation}\label{VAB}
V(X,Y)=X+Y+[X,A(X,Y)]+[Y,B(X,Y)].
\end{equation}
Moreover, $A(X,Y)$ and $B(X,Y)$ are expressible as normally convergent series of Lie brackets in $X$
and $Y$.

Let $\mathrm{ad}_X\colon A\to A$ denote the map $\mathrm{ad}_X(Z)=[X,Z]$ for all $Z\in A$.
Let $x=\mathrm{ad}_X$ and $y=\mathrm{ad}_Y$.   Define
\[
F(X,Y)=\frac{x}{e^{-x}-1}B(Y,X),\quad G(X,Y)=\frac{y}{e^y-1}A(Y,X).
\] 
(Our notation matches that of \cite[Sections 1.1--1.5]{rouviere}, with the following exception:  our $A(X,Y)$ and $B(X,Y)$  are $B(Y,X)$ and $A(Y,X)$ respectively,  in \cite[Section 1.5.2]{rouviere}.) From \eqref{VAB} we deduce that
\begin{equation}\label{KV1}
V(Y,X)=X+Y-(1-e^x)F(X,Y)-(e^y-1)G(X,Y).
\end{equation}
This is the first Kashiwara-Vergne equation.  
Since $\frac{e^{-x}-1}{x}$ and $\frac{e^y-1}{y}$ are  invertible  operators for $\|x\|,\|y\|<2\pi$ (hence, for $\|X\|,\|Y\| <\pi$), $F(X,Y)$ and $G(X,Y)$
are defined for $\|X\|,\|Y\|\leq \delta $ and are normally convergent series of Lie brackets on this domain. (More formally, we can first define $F(\delta \X,\delta \Y)$ and $G(\delta\X,\delta \Y)$
in the Banach-Lie algebra $\Li(\X,\Y)$ and then evaluate them at $X,Y\in A$ such that 
$\|X\|,\|Y\|\leq \delta$ via the assignments $\delta \X\mapsto X$ and $\delta \Y\mapsto Y$.)

Let us define  $V_t(X,Y)=\frac{1}{t}V(tX,tY)$ 
for all $X,Y\in A$ and $t\in \R$ non-zero and small enough (depending on $X$ and $Y$).  
Define $V_0(X,Y)=X+Y$. Let 
$v(t)=\mathrm{ad}_{V_t(X,Y)}=\log(e^{tx}e^{ty})$, which is defined also for small enough $t$. (We sometimes omit reference to $X$ and $Y$
to simplify notation.) Let us also define  
 $F_t(X,Y)=t^{-1}F(tX,tY)$ and $G_t(X,Y)=t^{-1}G(tX,tY)$.  Then \eqref{KV1} implies that
\begin{equation}\label{KV1t}
V_t(Y,X)=X+Y-(1-e^{tx})F_t(X,Y)-(e^{ty}-1)G_t(X,Y),
\end{equation}
for small enough $t$.

Let us introduce the partial differential operators: The operators $\partial_X V_t(X,Y)$ and $\partial_Y V_t(X,Y)$ act  on $Z\in A$
as follows:
\[
\partial_X V_t(X,Y):=\frac{d}{ds}V_t(X+sZ,Y)|_{s=0},\quad 
\partial_Y V_t(X,Y):=\frac{d}{ds}V_t(X,Y+sZ)|_{s=0}.
\]
The following formulas are given in \cite[Lemma 1.2]{rouviere} (where they are deduced from the formula for the  differential of the exponential map):
\begin{equation}\label{partialXY}
\begin{aligned}
\frac{e^{v(t)}-1}{v(t)}\partial_X V_t(X,Y) &=e^{-ty}\frac{1-e^{-tx}}{tx},\quad
\frac{e^{v(t)}-1}{v(t)}\partial_Y V_t(X,Y) &=\frac{1-e^{-ty}}{ty}.
\end{aligned}
\end{equation}

\begin{lemma}[Cf. \cite{rouviere}*{Proposition 1.3}]\label{dVtlemma}
We have
\begin{equation}\label{dVt}
\partial_t V_t(X,Y)=\partial_X V_t(X,Y)[X,F_t(X,Y)]+\partial_Y V_t(X,Y)[Y,G_t(X,Y)],
\end{equation}
for small enough $t$.
\end{lemma}
\begin{proof}
Throughout the proof we abbreviate  $V_t(X,Y)$, $F_t(X,Y)$, and $G_t(X,Y)$ to $V_t$, $F_t$ and $G_t$.
Taking $\partial_t$ in the definition of $V_t$ we get
\[
\partial V_t=-\frac{1}{t}V_t+\frac{1}{t}(\partial_X V_t)X+\frac 1 t(\partial_Y V_t)Y.
\]
Thus, we must show that
\[
-\frac{1}{t}V_t+\frac{1}{t}(\partial_X V_t)X+\frac 1 t(\partial_Y V_t)Y=
(\partial_X V_t)[X,F_t]+(\partial_Y V_t)[Y,G_t].
\]
Multiplying by $\frac{e^{v(t)}-1}{v(t)}$ on both sides of this equation and using \eqref{partialXY} we get
\[
\frac{1}{t}e^{-ty}+\frac 1 t Y-\frac 1 t V_t=\frac 1 t e^{ty}(1-e^{-tx})F_t+\frac 1 t (1-e^{-ty})G_t.
\]
Now multiplying by $te^{ty}$ on both sides and using that $te^{ty}V_t(X,Y)=V_t(Y,X)$ we get 
\eqref{KV1t}, which we know is valid. Working backwards we get the desired result. (To have invertibility of $\frac{e^{v(t)}-1}{v(t)}$ we need $\|v(t)\|<2\pi$ which 
can be arranged for small enough $t$.)
\end{proof}

For $f\in A$ invertible and $Z\in A$ we use the notation $f\cdot Z=fZf^{-1}$.

\begin{theorem}\label{fgODE}
There exists $\epsilon>0$ such that for any given Banach algebra $A$ there exist continuous functions 
$(X,Y)\mapsto R(X,Y)\in A$ and  $(X,Y)\mapsto S(X,Y)\in A$
defined for   $\|X\|\leq \epsilon$ and $\|Y\|\leq \epsilon$ 
such that $R(0,0)=S(0,0)=0$,
\[
V(e^{R(X,Y)}\cdot X,e^{S(X,Y)}\cdot Y)=X+Y
\]
for all $\|X\|\leq \epsilon$ and $\|Y\|\leq \epsilon$, and 
\begin{equation}
\begin{aligned}\label{RSgather}
R(X,Y) &=\frac Y 4 +[X,R'(X,Y)]+[Y,R''(X,Y)],\\
S(X,Y) &=-\frac X 4+[X,S'(X,Y)]+[Y,S''(X,Y)],
\end{aligned}
\end{equation}
for some $R'$, $R''$, $S'$, and $S''$.
Moreover, if $A$ is a $\Cstar$-algebra and $X$ and $Y$ are skewadjoints, then
$R$, $R'$, $R''$, $S$, $S'$, and $S''$ as above are all skewadjoints. 
\end{theorem}
\begin{proof}
We will work in the setting of the free Banach algebra $\A(\X,\Y)$ and its Banach-Lie subalgebra
$\Li(\X,\Y)$. 
We will construct $R_t,S_t\in \Li(\X,\Y)$ for $|t|\leq \epsilon$
such that $R_0=S_0=0$
and
\begin{equation}\label{flow}
V_t(e^{R_t}\cdot \X,e^{S_t}\cdot \Y)=\X+\Y,
\end{equation}
for all $|t|\leq \epsilon$. We will moreover show that $\lambda_\alpha R_t=R_{\alpha t}$
and $\lambda_\alpha S_t=S_{\alpha t}$ for all $|\alpha|\leq 1$.
To  derive  the theorem from this
we will make the assignment $\epsilon \X\mapsto X$, $\epsilon \Y\mapsto Y$ and then define $R(X,Y)=R_{\epsilon}(\frac 1 \epsilon X,\frac 1 \epsilon Y)$  and  $S(X,Y)=S_{\epsilon}(\frac 1 \epsilon X,\frac 1 \epsilon Y)$.

Notice that  \eqref{flow} holds for $t=0$ once $R_0=S_0=0$. 
Thus,
differentiating with respect to $t$,  \eqref{flow} is equivalent to
\begin{equation}\label{equivODE}
\partial_t V_t(\X_t,\Y_t)+\partial_X V_t(\X_t,\Y_t)
[\partial_t e^{R_t} e^{-R_t},\X_t]+
\partial_Y V_t(\X_t,\Y_t)[\partial_te^{S_t}e^{-S_t},\Y_t]=0,
\end{equation}
with $R_0=S_0=0$ (see \cite[pages 4--5]{rouviere}).
Here we have set $\X_t=e^{R_t}\cdot \X$ and $\Y_t=e^{S_t}\cdot \Y$.

Let $r_t=\ad_{R_t}$ and $s_t=\ad_{S_t}$. Consider the initial value problem
\begin{equation}\label{cauchyRS}
\begin{aligned}
\partial_t R_t &=\frac{r_t}{1-e^{-r_t}} F_t(\X,e^{-r_t}e^{s_t}\Y),\\
\partial_t S_t &=\frac{s_t}{1-e^{-s_t}}G_t(e^{-s_t}e^{r_t}\X,\Y),
\end{aligned}
\end{equation}
with initial conditions $R_0=S_0=0$.  The function
\[
(U,V,t)\stackrel{H}{\longmapsto} \frac{\ad_U}{1-e^{-\ad_U}} F_t(\X,e^{-\ad_U}e^{\ad_V}\Y)
\]
has the form $\sum_{n=1}^\infty t^{n-1}H_n(U,V)$ where 
\[
H_n(U,V)=\frac{\ad_U}{1-e^{-\ad_U}} F_n(\X,e^{-\ad_U}e^{\ad_V}\Y)\in \Li(\X,\Y)
\]
for sufficiently small $U$. (Recall that $F_n(\X,\Y)$ is a polynomial in $\X$  and $\Y$; here
$F(\X,\Y)=\sum_{n=1}^\infty F_n(\X,\Y)\in \Li[[\X,\Y]]$). Furthemore,
using the simple estimate 
$\|F_n(\X,T\Y)\|\leq (1+\|T\|)^n\|F_n(\X,\Y)\|$ for any operator $T$ acting on $\Li(\X,\Y)$ we derive that $\sum_{n=1}^\infty |t|^{n-1}\|H_n(U,V)\|<\infty$ for $(U,V,t)$
in a sufficiently small neighborhood of $(0,0,0)$. It follows that $H$ has uniformly bounded Frechet derivative for small enough $t$. The same can be said of the second equation in 
\eqref{cauchyRS}. This guarantees the existence and uniqueness of a solution to \eqref{cauchyRS} in the Banach space $\Li(\X,\Y)$ for $|t|<\epsilon$ and some $\epsilon>0$.

From \eqref{cauchyRS} and the formula for the differential of the exponential map we deduce that
\begin{equation}\label{cauchyeRS}
\begin{aligned}
\partial_t e^{R_t} &=F_t(e^{r_t}\X,e^{s_t}\Y)e^{R_t},\\
\partial_t e^{S_t} &=G_t(e^{r_t}\X,e^{s_t}\Y)e^{S_t}.
\end{aligned}
\end{equation}
With $R_t$ and $S_t$  that satisfy \eqref{cauchyeRS},  equation \eqref{equivODE} is 
equivalent to
\[
\partial_t V_t(\X_t,\Y_t)+\partial_X V_t(\X_t,\Y_t)
[F_t(\X_t,\Y_t),\X_t]+
\partial_Y V_t(\X_t,\Y_t)[G_t(\X_t,\Y_t),\Y_t]=0.
\]
But this is equation \eqref{dVt} from Lemma \ref{dVtlemma}, where $X$ and $Y$ haven been replaced by
$\X_t$ and $\Y_t$. This proves \eqref{flow}.

Let $0<\alpha\leq 1$. It is straightforward to check  that 
$(\lambda_\alpha R_{\frac t \alpha},\lambda_\alpha S_{\frac t \alpha})$ is also a solution of \eqref{cauchyRS} (cf. \cite[page 7]{rouviere} for the same verification for $f_t:=e^{R_t}$ and $g_t:=e^{S_t}$). 
It follows that $\lambda_\alpha R_t=R_{\alpha t}$ and $\lambda_\alpha S_t=S_{\alpha t}$.

Let us write $R_t=a(t)\X+b(t)\Y+\sum_{n\geq 2} R_{t,n}$, with $R_{t,n}\in \Li_n(\X,\Y)$.
From $\lambda_\alpha R_t=R_{\alpha t}$ we deduce that $a(t)$ and $b(t)$ are linear.
But from \eqref{cauchyRS} we see that $\partial_t R_t|_{t=0}=F_t(\X,\Y)|_{t=0}=\frac{\X}{4}$. Thus, $R_t=t\frac{\X}{4}+\sum_{n\geq 2} R_{t,n}$ for all $|t|\leq \epsilon$. Applying Lemma \ref{halfhalf} to $R_t$ we get
\[
R_t=t\frac{\X}{4}+[\X,R'_t]+[\Y,R''_t]
\]
for all $|t|\leq \frac{\epsilon}{2}$ and some $R_t',R_t''\in \Li(\X,\Y)$. We can derive similarly for $S_t$ that
\[
S_t=-t\frac Y 4+[\X,S'_t]+[\Y,S''_t]
\]
for $|t|\leq \frac{\epsilon}{2}$ and suitable $S'_t,S_t''\in \Li(\X,\Y)$. 
Let us relabel $\frac \epsilon 2$
as $\epsilon$ so that  these representations of $R_t$ and $S_t$ are valid for $|t|\leq \epsilon$.

To prove the theorem let us make the assignment $\epsilon\X\mapsto X$,
$\epsilon \Y\mapsto Y$, which, for $\|X\|,\|Y\|\leq \epsilon$, extends to
a Banach algebra contractive homomorphism from $\A(\X,\Y)$ to $A$. Then 
$R(X,Y)=R_{\epsilon}(\frac \X\epsilon,\frac\Y\epsilon)$ and  
$S(X,Y)=S_{\epsilon}(\frac \X\epsilon,\frac\Y\epsilon)$ are as desired.

Finally, suppose that $A$ is a $\Cstar$-algebra and that $X$ and $Y$ are skewadjoints.
Since $\sigma Z=-Z^*$ is a continuous Lie algebra homomorphism and $R(X,Y)$ is a convergent series of Lie brackets on $X,Y$ we get that $\sigma R(X,Y)=R(\sigma X,\sigma Y)=R(X,Y)$, i.e., $R(X,Y)$ is skewadjoint.  The same argument applies to $S(X,Y)$ and to 
$R', R'',S',S''$.
\end{proof}

\end{document}